\documentclass[reqno,twoside,11pt]{amsart}

\theoremstyle{plain}

\def\endproof{\hspace*{\fill}\mbox{\ \rule{.1in}{.1in}}\medskip }

\numberwithin{equation}{section}

\usepackage{amsmath, amsfonts, amssymb, esint, amsthm, nicefrac, bbm}
\usepackage{booktabs}
\usepackage{epsfig, graphicx}
\usepackage{accents, csquotes} 

\newcommand{\heis}{\mathbb{H}}
\newcommand{\dd}{~\mbox{d}}
\newcommand{\osc}{\mbox{osc}}
\newcommand{\p}{\mathrm{\bf p}}
\newcommand{\q}{\mathrm{\bf q}}
\newcommand{\s}{\mathrm{\bf s}}
\newcommand{\R}{\mathbb{R}}
\newcommand{\D}{\mathcal{D}}

\theoremstyle{plain}

\newtheorem{Teo}{Theorem}[section]
\newtheorem{Lemma}[Teo]{Lemma}
\newtheorem{Cor}[Teo]{Corollary}

\theoremstyle{definition}
\newtheorem{Rem}[Teo]{Remark}
\newtheorem{Def}[Teo]{Definition}

\begin{document}
\title[Noisy Tug of War games for the $p$-Laplacian]{Noisy
  Tug of War games for the $\mathbf{p}$-Laplacian: ${\mathbf{1<p<{\boldsymbol \infty}}}$}
\author{Marta Lewicka}
\address{University of Pittsburgh, Department of Mathematics, 
139 University Place, Pittsburgh, PA 15260}
\email{lewicka@pitt.edu} 

\begin{abstract}
We propose a new finite difference approximation to the
Dirichlet problem for the homogeneous $\mathbf{p}$-Laplace equation posed on
an $N$-dimensional domain, in connection with the 
Tug of War games with noise. Our
game and the related mean-value expansion that we develop,
superposes the ``deterministic averages'' ``$\frac{1}{2}(\inf
+\sup)$'' taken over balls, with
the ``stochastic averages'' ``$\fint$'', taken over $N$-dimensional
ellipsoids whose aspect ratio depends on $N,\mathbf{p}$ 
and whose orientations span all directions while determining
$\inf / \sup$. We show that the unique
solutions $u_\epsilon$ of the related dynamic programming principle
are automatically continuous for continuous boundary 
data, and coincide with the well-defined game
values. Our game has thus the min-max property: the order of
supremizing the outcomes over strategies of one player and
infimizing over strategies of their opponent, is immaterial. We
further show that domains satisfying the exterior
corkscrew condition are game regular in this context, i.e. the family
$\{u_\epsilon\}_{\epsilon\to 0}$ converges uniformly to the unique
viscosity solution of the Dirichlet problem.  
\end{abstract}

\date{September 2, 2019}
\maketitle

\section{Introduction}

In this paper, we study the finite difference approximations to the
Dirichlet problem for the homogeneous {\em $\p$-Laplace equation} $\Delta_\p 
u=0$, posed on an $N$-dimensional domain, in connection to the dynamic programming principles
of the so-called {\em Tug of War games with noise}.

It is a well known fact that for $u\in\mathcal{C}^2(\R^N)$ there holds
the following {\em mean value expansion}:
$$\fint_{B(x,r)} u(y)\dd y = u(x) + \frac{r^2}{2(N+2)}\Delta u(x) +
o(r^2)\qquad \mbox{as }\; r\to 0+.$$
Indeed, an equivalent condition for harmonicity $\Delta u =0$ is
the mean value property, and thus $\Delta u(x)$
provides the second-order offset from the satisfaction of this
property. When we replace $B(x,r)$ by an ellipse $E(x,r;\alpha, \nu)=x+\{y\in \R^N;
~\langle y,\nu\rangle^2 + \alpha^2|y-\langle
y,\nu\rangle\nu|^2<\alpha^2r^2\}$ with the radius $r$, the aspect ratio $\alpha>0$ and
oriented along some given unit vector $\nu$, we obtain:
\begin{equation}\label{ellip}
\fint_{E(x,r;\alpha,\nu)} u(y)\dd y = u(x) + \frac{r^2}{2(N+2)}\Big(\Delta u(x) +
(\alpha^2-1)\langle \nabla^2u(x) : \nu^{\otimes 2}\rangle \Big) + o(r^2).
\end{equation}
Recalling the interpolation:
\begin{equation}\label{MV1} 
\Delta_\p u = |\nabla u|^{\p-2}\big(\Delta u + (\p-2) \Delta_\infty u\big), 
\end{equation} 
the formula (\ref{ellip}) becomes: $\fint_{E(x,r;\alpha,\nu)} u(y)\dd y = u(x)
+\frac{r^2 |\nabla u|^{2-\p}}{2(N+2)}\Delta_\p u(x) + o(r^2),$
for the choice  $\alpha = \sqrt{\p-1}$ and $\nu = \frac{\nabla u(x)}{|\nabla u(x)|}$. 
To obtain the mean value expansion where the left hand side averaging
does not require the knowledge of $\nabla u(x)$ and
allows for the identification of a $\p$-harmonic function that is a
priori only continuous,  we need to, in a sense, additionally average over all
equally probable vectors $\nu$. This can be carried out  by superposing: 
\begin{itemize}
\item[(i)]  the {\em deterministic average} ``$\frac{1}{2}(\inf +\sup)$'', with
\item[(ii)] the {\em stochastic average}
``$\fint$'', taken over appropriate ellipses $E$ whose aspect ratio depends on $N,\p$
and whose orientations $\nu$ span all directions while determining $\inf / \sup$ in (i).
\end{itemize}

In fact, such construction can be made precise (see Theorem
\ref{th_expansion}), leading to the expansion:
\begin{equation}\label{el2}
\begin{split}
\frac{1}{2}\Big(\inf_{z\in B(x,r)} + \sup_{z\in B(x,r)} \Big)
& \fint_{\begin{minipage}{4.4cm} $E\big(z,\gamma_\p r,
      \alpha_\p(\big|\frac{z-x}{r}\big|), \frac{z-x}{|z-x|}\big)$\end{minipage}} 
u(y)\dd y \\ & \qquad\qquad \qquad\qquad = u(x) + \frac{\gamma_\p^2 r^2|\nabla u|^{\p-2}}{2(N+2)}\Delta_\p u(x) + o(r^2),
\end{split}
\end{equation}
with $\gamma_\p $ that is a fixed stochastic
sampling radius factor, and with $\alpha_\p$ that is the aspect ratio
in radial function of the deterministically chosen position $z\in
B(x,r)$. The value of $\alpha_\p$ varies quadratically from $1$ at the center of 
$B(x,r)$ to $a_\p$ at its boundary, where $a_\p$ and $\gamma_\p$ satisfy the compatibility
condition $\frac{N+2}{\gamma_\p^2} + a_\p^2=\p-1$.

\medskip

We will be concerned with the mean value expansions of the form (\ref{el2}), in connection with
the specific Tug of War games with noise. This connection has been 
displayed in \cite{PS} by Peres and Scheffield, based on another interpolation property of $\Delta_\p$:
\begin{equation}\label{MV2}
\Delta_\p u = |\nabla u|^{\p-2}\Big(|\nabla u|\Delta_1 u + (\p-1)\Delta_\infty u\Big),
\end{equation}
which has first appeared (in the context of the applications of
$\Delta_\p$ to image recognition) in \cite{Kaw}.
Indeed, the construction in \cite{PS} interpolates from: {(i)} the $1$-Laplace operator $\Delta_1$ corresponding to
the motion by curvature game studied by Kohn and Serfaty \cite{KS},
to { (ii)} the $\infty$-Laplacian  $\Delta_\infty$ corresponding to the pure Tug of War studied by
Peres, Schramm, Scheffield and Wilson \cite{PSSW}. We remark that if one
uses (\ref{MV1}) instead of (\ref{MV2}), one is lead to the  games studied by Manfredi, Parviainen and
Rossi \cite{MPR}, that interpolate from: $\Delta_2$ (classically corresponding \cite{Dbook}
to Brownian motion), to $\Delta_\infty$; this approach however poses a
limitation on the exponents $\p\in [2,\infty)$.

The original game presented in \cite{PS} was a two-player, zero-sum game, stipulating that at each turn,
position of the token is shifted by some vector $\sigma$ within the prescribed radius
$r=\epsilon>0$, by a player who has  won the coin toss, which is followed
by a further update of the position by a random ``noise vector''. The noise  vector is uniformly distributed
on the codimension-$2$ sphere, centered at the current position,
contained within the hyperplane that is orthogonal to the last
player's move $\sigma$, and with radius proportional to
$|\sigma|$ with  factor $\gamma=\sqrt{\frac{N-1}{\p-1}}$.
We again interpret that $\gamma$ interpolates from: { (i)} $+\infty$ at the critical exponent
$\p=1$ that corresponds to choosing a direction line and subsequently
determining its orientation, to { (ii)} $0$ at the critical exponent $\p=\infty$ that corresponds to not adding the
random noise at all.

\medskip

In this paper we utilize the full $N$-dimensional sampling on
ellipses $E$, rather than on spheres. Together with another
modification taking into account the boundary data $F$, we achieve that the
solutions of the dynamic programming principle at each scale
$\epsilon >0$ are automatically continuous (in fact, they inherit
the regularity of $F$) and coincide with the well-de\-fi\-ned game
values. The game stops almost surely under the additional
compatibility condition, displayed in (\ref{prope2}), on the scaling factors $\gamma_\p$
and $a_\p$ this condition is viable for any $\p\in (1,\infty)$. 
Our game has then the min-max property: the order of
supremizing the outcomes over strategies of one player and
infimizing over strategies of his opponent, is immaterial. 

This point has been left unanswered in the case of the game in \cite{PS},
where the regularity (even measurability) of the possibly distinct game values was not clear.
We also point out that in \cite{AHP}, the authors presented
a variant of the \cite{PS} game where the deterministic / stochastic
sampling takes place, respectively on: $(N-1)$-dimensional 
spheres, and $(N-1)$-dimensional balls within the orthogonal hyperplanes.
They obtain the min-max property and continuity of solutions to the
mean value expansion in their setting, albeit at the expense of much more complicated
analysis, passing through measurable construction and comparison to
game values. In our case the uniqueness and continuity follow directly,
much like in the linear $\p=2$ case where the $N$-dimensional
averaging guarantees smoothness of harmonic functions.

\medskip

\subsection{The content and structure of the paper}
In section 2, Theorem \ref{th_expansion}, we prove the validity of our
main mean value expansion (\ref{el2}). In the following remarks we show how
other expansions (with a wider range of exponents, with sampling set degenerating at the boundary,
or pertaining to the \cite{PS} codimension-$2$ sampling) arise in the same analytical context. 

In Theorem \ref{thD} in section 3, we obtain the existence, uniqueness
and regularity of solutions $u_\epsilon$ to the dynamic programming
principle (\ref{DPP2}) at each sampling scale $\epsilon$, that can be
seen as a finite difference approximation to
the $\p$-Laplace Dirichlet problem with continuous boundary data $F$.
In particular, each $u_\epsilon$ is continuous up to the boundary,
where it assumes the values of $F$.
Then, in Theorem \ref{conv_nondegene} we show that in case $F$ is already
a restriction of a $\p$-harmonic function with non-vanishing gradient,
the corresponding family $\{u_\epsilon\}_{\epsilon\to 0}$ uniformly converges to $F$ at the
rate that is of first order in $\epsilon$. Our proof uses an analytical
argument and it is based on the observation that for $s$ sufficiently
large, the mapping $x\mapsto |x|^s$ yields the variation that pushes the
$\p$-harmonic function $F$ into the region of $\p$-subharmonicity. In
the linear case $\p=2$, the quadratic correction $s=2$ suffices,
otherwise we give a lower bound (\ref{s_bound}) for the admissible
exponents $s=s(\p, N, F)$.

In section 4, we develop the probability setting for the Tug of War
game modelled on (\ref{el2}) and (\ref{DPP2}). In Lemma \ref{facom} we
show that the game stops almost surely if the scaling factors $\gamma_\p,
a_\p$ are chosen appropriately.
Then in Theorem \ref{are_equal}, using a classical martingale
argument, we prove that our game has a value, coinciding with the unique, continuous
solution $u_\epsilon$.

In section 5 we address convergence of the family
$\{u_\epsilon\}_{\epsilon\to 0}$. In view of its equiboundedness, it suffices to prove
equicontinuity. We first observe, in Theorem \ref{transfer_p},
that this property is equivalent to the seemingly weaker property of equicontinuity at the
boundary. Again, our argument is analytical rather than probabilistic,
based on the translation and well-posedness of (\ref{DPP2}). We then define
the {\em game regularity} of the boundary points, which turns out to be a notion equivalent to the
aforementioned boundary equicontinuity. Definition
\ref{def_gamereg}, Lemma \ref{thm_gamenotreg_noconv} and Theorem
\ref{thm_gamereg_conv} mimic the parallel statements in \cite{PS}. We further prove in Theorem
\ref{unif-topharm} that any limit of a converging sequence in 
$\{u_\epsilon\}_{\epsilon\to 0}$ must be the viscosity solution to the
$\p$-harmonic equation with boundary data $F$. By uniqueness of such
solutions, we obtain the uniform convergence of the entire family $\{u_\epsilon\}_{\epsilon\to 0}$ in
case of the game regular boundary.

In section 6 we show that domains that satisfy the {\em exterior
corkscrew condition} are game regular. The proof in Theorem \ref{corkthengamereg}
is based on the concatenating strategies technique and the annulus
walk estimate taken from \cite{PS}. We expand the proofs and carefully
provide the details omitted in \cite{PS}, for the benefit of the
reader less familiar with the probability techniques.

\medskip

Finally, let us mention that similar results and approximations,
together with their game-theoretical interpretation, can be also developed in other
contexts, such as: the obstacle problems, nonlinear potential theory in
Heisenberg group (or in other subriemannian
geometries), Tug of War on graphs, the non-homogeneous problems, 
problems with non-constant coefficient $\p$, and the fully nonlinear case of $\Delta_\infty$. 

\medskip

\subsection{Notation for the $\p$-Laplacian} 
Let $\mathcal{D}\subset \mathbb{R}^N$ be an open, bounded, connected
set. Given $\p\in (1,\infty)$, consider the following Dirichlet integral:
$$\mathcal{I}_\p(u) = \int_{\D} |\nabla u(x)|^\p~\dd x \qquad
\mbox{for all }~ u\in W^{1,\p}(\D),$$
that we want to minimize among all functions $u$
subject to some given boundary data. The condition for the vanishing of the
first variation of $\mathcal{I}_\p$, 
assuming sufficient regularity of $u$ so that the divergence theorem
may be used,  takes the form: 
$$\int_{\D} \eta\, \mathrm{div}\big(|\nabla u|^{\p-2}\nabla u\big)\dd x= 0
\qquad \mbox{for all } \eta\in\mathcal{C}_c^\infty(\D),$$ 
which, by the fundamental theorem of Calculus of Variations, yields:
\begin{equation}\label{pla}
\Delta_\p u = \mathrm{div}\Big(|\nabla u|^{\p-2}\nabla u\Big) = 0 \quad\mbox{in } \D.
\end{equation}
The operator in (\ref{pla}) is
called the {\em $\p$-Laplacian}, the partial differential equation
(\ref{pla}) is called the {\em $\p$-harmonic equation} and its solution
$u$ is a {\em $\p$-harmonic function}. It is not hard to compute:
\begin{equation*}\label{added_calc}
\begin{split}
\Delta_\p u & = |\nabla u|^{\p-2}\Delta u + \left\langle \nabla\big(|\nabla
u|^{\p-2}\big), \nabla u\right\rangle = |\nabla u|^{\p-2}\Big(\Delta u + (\p-2)
\Big\langle \nabla^2u: \Big(\frac{\nabla u}{|\nabla u|}\Big)^{\otimes 2}\Big\rangle\Big),
\end{split}
\end{equation*}
which is precisely (\ref{MV1}).
The second term in parentheses is called the {\em $\infty$-Laplacian}:
\begin{equation*}\label{infla}
\Delta_\infty u= \Big\langle \nabla^2u : \Big(\frac{\nabla u}{|\nabla u|}\Big)^{\otimes 2}\Big\rangle.
\end{equation*}
The notation above refers to taking the scalar (Frobenius) product of
the $N\times N$ matrix $\nabla^2 u$ with the rank-1 matrix: 
$\big(\frac{\nabla u}{|\nabla u|}\big)^{\otimes 2} 
= \frac{\nabla u}{|\nabla u|}\otimes \frac{\nabla u}{|\nabla u|} =
\big(\frac{\nabla u}{|\nabla u|}\big)\big(\frac{\nabla u}{|\nabla
  u|}\big)^T$. Using the scalar product of vectors notation, this is
equivalent to writing: $\Delta_\infty u= \big\langle \nabla^2u
\frac{\nabla u}{|\nabla u|}, \frac{\nabla u}{|\nabla u|}\big\rangle$. 

\smallskip

\noindent Applying (\ref{MV1}) to the effect that  $\Delta_1u= |\nabla
u|^{-1}\big(\Delta u -\Delta_\infty u\big)$ and introducing it in
(\ref{MV1}) again, yields (\ref{MV2}). Likewise,  for every $1<\p<\q<\s<\infty$ there holds:
$$(\s-\q)|\nabla u|^{2-\p}
\Delta_\p u = (\s-\p)|\nabla u|^{2-\q} \Delta_\q u + (\p-\q)|\nabla u|^{2-\s} \Delta_\s u.$$

\medskip

\subsection{Acknowledgments}
The author is grateful to Yuval Peres for discussions about Tug of War games. 
Support by the NSF grant DMS-1613153 is acknowledged.

\section{A mean value expansion for $\Delta_{\p}$} 

For $\rho, \alpha>0$ and a unit vector $\nu\in\R^N$, we denote by
$E(0,\rho;\alpha,\nu)$ the ellipse centered at $0$, with
radius $\rho$, and with aspect ratio $\alpha$ oriented along $\nu$, namely:
$$E(0,\rho;\alpha,\nu)=\big\{y\in\R^N;~ \frac{\langle
  y,\nu\rangle^2}{\alpha^2}+|y-\langle y, \nu\rangle\nu|^2<\rho^2\big\}.$$
For $x\in\R^N$, we have the translated ellipse:
$$E(x,\rho;\alpha,\nu) = x+ E(0,\rho;\alpha,\nu).$$
Note that, when $\nu=0$, this formula also makes sense and returns the
ball $E(x,\rho;\alpha,0) = B(x,\rho)$ centered at $x$ and with radius $\rho$.

Given a continuous function $u:\mathbb{R}^N\to\mathbb{R}$,  define the
averaging operator:
\begin{equation*}
\begin{split} 
\mathcal{A}\big(u; \rho,\alpha,\nu\big)(x) & =
\fint_{E(x,\rho;\alpha,\nu)}u(y)\dd y = \fint_{B(0,1)}
u\big(x+\rho y+\rho (\alpha-1)\langle y,\nu\rangle\nu\big)\dd y.
\end{split}
\end{equation*}
In what follows, we will often use the above linear change of variables:
$$B(0,1)\ni y\mapsto \rho\alpha\langle y,\nu\rangle\nu + \rho\big(y-\langle
y, \nu\rangle \nu\big)\in E(0,\rho;\alpha, \nu).$$

\begin{Teo}\label{th_expansion}
Given $\p\in (1,\infty)$, fix
any pair of scaling factors $\gamma_\p, a_\p>0$ such that:
\begin{equation}\label{prope}
\frac{N+2}{\gamma_\p^2} +a_\p^2 = \p-1.
\end{equation}
Let $u\in\mathcal{C}^2(\mathbb{R}^N)$. Then, for every $x_0\in\R^N$ such that $\nabla u(x_0)\neq 0$, we have:
\begin{equation}\label{expansion}
\begin{split}
\frac{1}{2}\Big(\inf_{x\in B(x_0, r)} +& \sup_{x\in B(x_0, r)} \Big) \mathcal{A}\Big(u; \gamma_\p r,
1+(a_\p-1)\frac{|x-x_0|^2}{r^2},\frac{x-x_0}{|x-x_0|}\Big)(x)  
\\ & \quad = u(x_0) +  \frac{\gamma_\p^2r^2}{2(N+2)} |\nabla u(x_0)|^{2-\p} 
\Delta_\p u(x_0) +o(r^2)\qquad\mbox{ as }\;r\to 0+.
\end{split}
\end{equation}
The coefficient in the rate of convergence $o(r^2)$ depends only on
$\p$, $N$, $\gamma_\p$ and (in increasing manner) on
$|\nabla u(x_0)|$, $|\nabla^2u(x_0)|$ and the modulus of continuity of $\nabla^2 u$ at $x_0$.
\end{Teo}

\bigskip

The expression in the left hand side of the formula (\ref{expansion})
should be understood as the deterministic average $\frac{1}{2}
(\inf +\sup)$, on the ball $B(x_0, r)$, of the function $x\mapsto f_u(x;x_0,r)$ in:
\begin{equation}\label{f}
\begin{split}
f_u(x; x_0, r) & = \mathcal{A}\Big(u; \gamma r,
1+(a-1)\frac{|x-x_0|^2}{r^2},\frac{x-x_0}{|x-x_0|}\Big)(x)\\ & =\fint_{B(0,1)}
u\big(x+\gamma r y + \frac{\gamma (a-1)}{r}\langle y, x-x_0\rangle (x-x_0)\big)\dd y,
\end{split}
\end{equation}
where $\gamma=\gamma_\p$ and $a=a_\p$. We will frequently use the notation:
\begin{equation}\label{aveS}
S_r u(x_0) = \frac{1}{2}\Big(\inf_{x\in B(x_0, r)} +\sup_{x\in B(x_0, r)} \Big)  f_u(x; x_0, r).
\end{equation}
For each $x\in B(x_0, r)$ the integral quantity in (\ref{f}) returns the average of
$u$ on the $N$-dimensional ellipse centered at $x$, with radius $\gamma r$, 
and with aspect ratio $\frac{r^2+(a-1)|x-x_0|^2}{r^2}$ along the
orientation vector $\frac{x-x_0}{|x-x_0|}$. Equivalently, writing
$x=x_0+ry$, the value $f_u(x; x_0, r)$ is the average of $u$ on the
scaled ellipse $x_0+ rE\big(y,\gamma; 1+(a-1)|y|^2,\frac{y}{|y|}\big)$.
Since the aspect ratio
changes smoothly from $1$ to $a$ as $|x-x_0|$ decreases from $0$ to
$r$, the said ellipse coincides with the ball $B(x_0,
\gamma r)$ at $x=x_0$ and it interpolates as $|x-x_0|\to r-$, to
$E\big(x,\gamma r; a, \frac{x-x_0}{|x-x_0|}\big).$ 

\begin{figure}[htbp]
\centering
\includegraphics[scale=0.3]{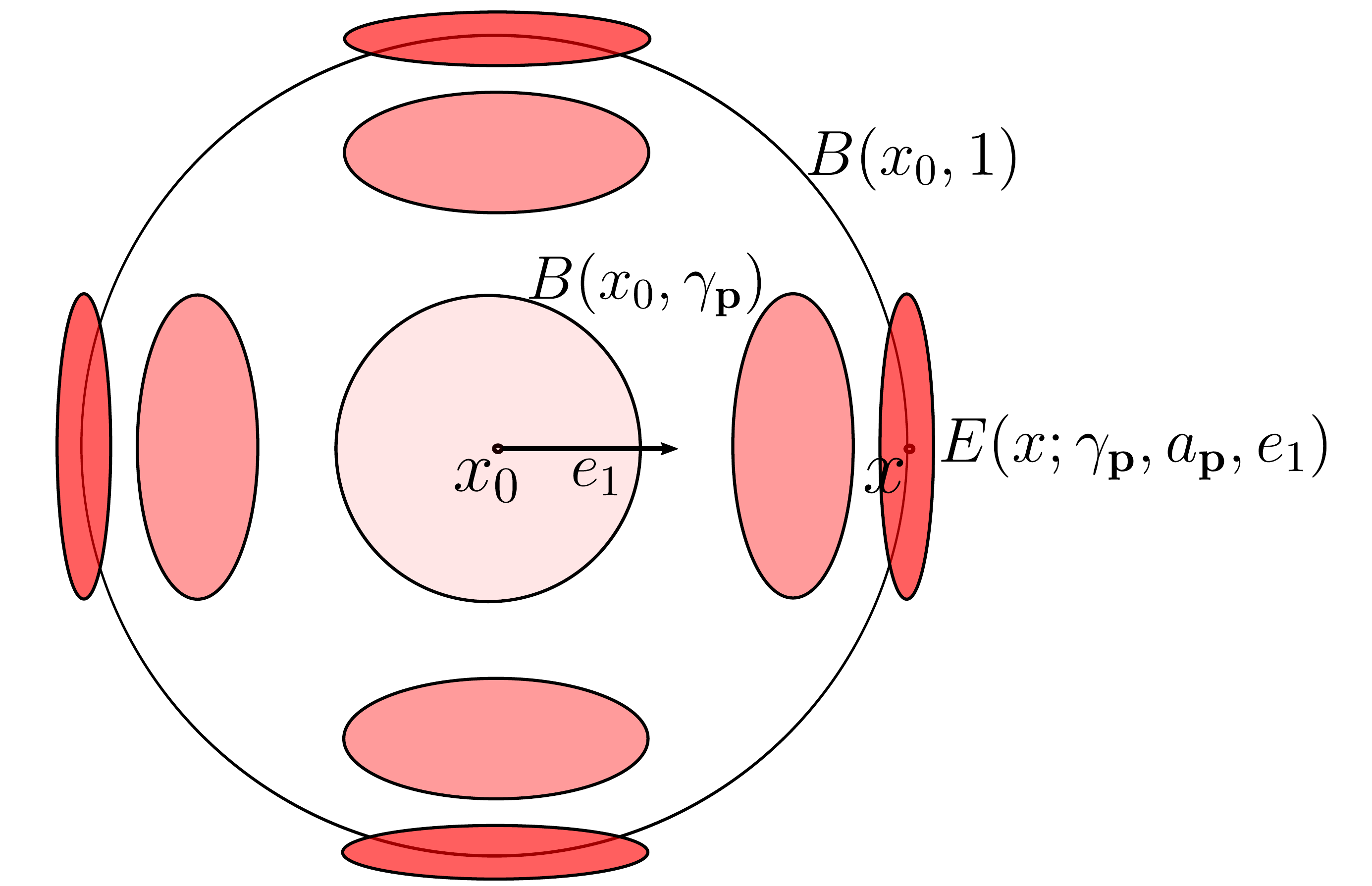}
    \caption{{The sampling ellipses in the expansion
        (\ref{expansion}), when $r=1$.}}
\label{f:exp_proof}
\end{figure}

\bigskip

\newpage

\noindent {\bf Proof of Theorem \ref{th_expansion}.}

\smallskip

\noindent {\bf 1.} We fix $\gamma, a> 0$ and consider the Taylor expansion of $u$ 
at a given $x\in B(x_0, \rho)$ under the second
integral in (\ref{f}). Observe that the first order increments
are linear in $y$, hence they integrate to $0$ on $B(0,1)$. These
increments are of order $r$ and we get:
\begin{equation}\label{f_approx}
\begin{split}
&f_u(x; ~x_0, r)  \\ & = u(x) + \frac{1}{2}\Big\langle \nabla^2 u(x): \fint_{B(0,1)} \Big(\gamma
r y + \frac{\gamma(a-1)}{r}\langle y, x-x_0\rangle (x-x_0)\Big)^{\otimes 2}\dd
y\Big\rangle + o(r^2)\\ & = u(x) + \frac{\gamma^2}{2}\Big\langle
\nabla^2 u(x): r^2 \fint_{B(0,1)} y^{\otimes 2} \dd y +2 (a-1)\fint_{B(0,1)} \langle y,
x-x_0\rangle y\dd y \otimes (x-x_0) \\ & \qquad \qquad\qquad\qquad
\qquad + \frac{(a-1)^2}{r^2}\Big(\fint_{B(0,1)} \langle y,
x-x_0\rangle^2\dd y\Big) (x-x_0)^{\otimes 2}\Big\rangle + o(r^2).
\end{split}
\end{equation}
Recall now that:
$$\fint_{B(0,1)}y^{\otimes 2}\dd y = \Big( \fint_{B(0,1)} y_1^2\dd y \Big) Id_N = \frac{1}{N+2} Id_N.$$
Consequently, (\ref{f_approx}) becomes:
\begin{equation*}
\begin{split} 
f_u(x; x_0, r) & = u(x)  + \frac{\gamma^2 r^2}{2(N+2)}\Delta
u(x) \\ & \quad + \frac{\gamma^2(a-1)}{2}\Big(\frac{2}{N+2}+\frac{(a-1)|x-x_0|^2}{r^2(N+2)}
\Big)\big\langle \nabla^2u(x) :  
(x-x_0)^{\otimes 2}\big\rangle + o(r^2)\\ & = \bar f_u(x; x_0, r) + o(r^2),
\end{split}
\end{equation*}
where a further Taylor expansion of $u$ at $x_0$ gives:
\begin{equation*}
\begin{split}
\bar f_u(x; x_0, r)  = u(x_0) & + \langle \nabla u(x_0), x-x_0\rangle
+ \frac{\gamma^2 r^2}{2(N+2)}\Delta
u(x_0) \\ & + \Big(\frac{1}{2} +
\frac{\gamma^2(a-1)}{2}\Big(\frac{2}{N+2} + \frac{(a-1)|x-x_0|^2}{r^2(N+2)}\Big)\Big) \big\langle \nabla^2u(x_0) : 
(x-x_0)^{\otimes 2}\big\rangle.
\end{split}
\end{equation*}

The left hand side of (\ref{expansion}) thus satisfies:
\begin{equation}\label{f_approx2}
\begin{split}
\frac{1}{2}\Big(\inf_{x\in B(x_0,r)}& f_u(x; x_0, r)  + \sup_{x\in
  B(x_0,r)}f_u(x; x_0, r) \Big) \\ & = \frac{1}{2}\Big(\inf_{x\in B(x_0,r)}
\bar f_u(x; x_0, r) + \sup_{x\in B(x_0,r)} \bar f_u(x; x_0, r) \Big) + o(r^2),
\end{split}
\end{equation}
Since on $B(x_0,r)$ we have: $\bar f_u(x; x_0,r) = u(x_0) + \langle \nabla
u(x_0), x-x_0\rangle + O(r^2)$, the assumption $\nabla u(x_0)\neq
0$ implies that the continuous function $\bar f_u(\cdot \;; x_0,r) $
attains its extrema on the boundary $\partial B(x_0,r)$, provided that
$r$ is sufficiently small. This reasoning justifies that $\bar f_u$ in (\ref{f_approx2}) may be
replaced by the quadratic polynomial function:
\begin{equation*}
\begin{split}
\bar{\bar f}_u(x; x_0, r)  = u(x_0) & + \frac{\gamma^2 r^2}{2(N+2)}\Delta
u(x_0) + \langle \nabla u(x_0), x-x_0\rangle
\\ & + \Big(\frac{1}{2} + \frac{\gamma^2(a^2-1)}{2(N+2)}\Big)\big\langle \nabla^2u(x_0) : 
(x-x_0)^{\otimes 2}\big\rangle.
\end{split}
\end{equation*}

\medskip

{\bf 2.} We now argue that $\bar{\bar f}_u$ attains its extrema on
$\bar B(x_0,r)$, up to error $O(r^3)$ whenever $r$ is
sufficiently small, precisely at the opposite
boundary points $x_0+ r\frac{\nabla u(x_0)}{|\nabla u(x_0)|}$ and
$x_0-r\frac{\nabla u(x_0)}{|\nabla u(x_0)|}$. We recall the adequate
argument from \cite{PS}, for the convenience of the 
reader. After translating and rescaling, it suffices to investigate 
the extrema on $\bar B(0,1)$, of the functions:
$$g_r(x) = \langle a, x\rangle + r\langle A: x^{\otimes 2}\rangle,$$
where $a\in \R^N$ is of unit length and $A\in \R^{N\times N}_{\mathrm{sym}}$.
Fix a small $r>0$ and let $x_{max}\in \partial B(0,1)$ be a maximizer
of $g_r$. Writing $g_r(x_{max})\geq g_r(a)$, we obtain:
\begin{equation*}
\begin{split}
\langle a, x_{max}\rangle & \geq 1 + r \big\langle A
: a^{\otimes 2} - x_{max}^{\otimes 2} \big\rangle \\ & \geq 1 - r |A|
\big| a^{\otimes 2} - x_{max}^{\otimes 2} \big| \geq 1 - 2r |A| |a - x_{max}|.
\end{split}
\end{equation*}
Thus there holds: $ |a - x_{max} |^2 = 2 - 2\langle a, x_{max}\rangle
\leq 4 r |A| |a - x_{max}|$, and so finally:
$$ |a- x_{max} | \leq 4r |A|.$$
Since  $\langle a, x_{max}\rangle \leq 1$, we conclude that:
\begin{equation*}
\begin{split}
0\leq g_r(x_{max}) - g_r(a) & = \langle a, x_{max}- a\rangle +
r \big\langle A : x_{max}^{\otimes 2} - a^{\otimes 2}  \big\rangle \\ & \leq
r \big\langle A : x_{max}^{\otimes 2} - a^{\otimes 2}  \big\rangle \leq
2 r |A| |a - x_{max}| \leq 8r^2 |A|^2.
\end{split}
\end{equation*}

Likewise, for a minimizer $x_{min}$ we have:
$ 0\geq g_r(x_{min}) - g_r(-a)  \geq  -8r^2 |A|^2$. It follows that:
\begin{equation*}
\begin{split}
\big| \frac{1}{2}\big(g_r(x_{min}) + &g_r (x_{max})\big) - \frac{1}{2}
\big(g_r(-a) + g_r (a)\big) \big| \leq  8 r^2 |A|^2,
\end{split}
\end{equation*}
which proves the claim for the unscaled functions $\bar{\bar f}_u$.

\medskip 

{\bf 3.} We now observe that for $\gamma=\gamma_\p$, $a=a_\p$
satisfying (\ref{prope}), there holds:
\begin{equation}\label{final}
\begin{split}
\frac{1}{2}\Big(&\inf_{x\in B(x_0,r)} \bar{\bar f}_u(x; x_0, r)  + \sup_{x\in
  B(x_0,r)}\bar{\bar f}_u(x; x_0, r) \Big) \\ & = u(x_0)  + \frac{\gamma^2 r^2}{2(N+2)}\Delta
u(x_0) + r^2 \Big(\frac{1}{2} + \frac{\gamma^2(a^2-1)}{2(N+2)}
\Big)\Delta_\infty u(x_0) + O(r^3) \\ &  = u(x_0)  + \frac{\gamma^2
  r^2}{2(N+2)}\Big( \Delta u(x_0) +  \Big(\frac{N+2}{\gamma^2} +
a^2 -1\Big)\Delta_\infty u(x_0) \Big) + O(r^3)\\ & =  u(x_0)  + \frac{
  \gamma_\p^2 r^2}{2(N+2)}\Big( \Delta u(x_0) +  (\p -2)\Delta_\infty u(x_0)
\Big) + O(r^3) \\ & =  u(x_0)  + \gamma_\p^2r^2 \frac{|\nabla u(x_0)|^{2-\p}
}{2(N+2)} \Delta_\p u(x_0) + O(r^3),
\end{split}
\end{equation}
where in the last step we used (\ref{MV1}). This completes the proof in view of (\ref{f_approx2}).
\endproof

\begin{Rem}\label{ag-general}
A few heuristic observations are in order.
When $\p\to\infty$, one can take $a_\p=1$ and
$\gamma_\p\sim 0$  in (\ref{prope}), whereas (\ref{expansion}) is
linked  to the following well-known
expansion and to the absolutely minimizing Lipschitz extension
property of the infinitely harmonic functions:
$$\frac{1}{2}\Big(\inf_{B(x_0, r)} u + \sup_{B(x_0, r)}u \Big) = u(x_0)  + \frac{r^2}{2}
\Delta_\infty u(x_0) +o(r^2).  $$
When $\p=2$, then choosing $a_\p=1$ and $\gamma_\p\sim \infty$
corresponds to taking both stochastic and deterministic averages on
balls, whose radii have ratio $\sim\infty$. Equivalently, one may average
stochastically on $B(x_0,r)$ and deterministically on $B(x_0,0)\sim
\{x_0\}$, consistently with another familiar expansion:
$$\mathcal{A}\big(u;r,1,0\big) (x_0) = \fint_{B(x_0,r)} u(y)\dd y = u(x_0) +
\frac{r^2}{2(N+2)}\Delta_2u(x_0) + o(r^2).$$
On the other hand, when $\p\to 1+$, then there must be $a_\p\to 0+$
and the critical choice $a_\p=0$  is the only one valid for
every $\p\in (1, \infty)$. It corresponds to varying the aspect ratio along the radius of $B(x_0, r)$
from $1$ to $0$ rather than to $a_\p>0$, and taking
the stochastic averaging domains to be the corresponding ellipses:
$$E\big(x,\gamma r; 1-\frac{|x-x_0|^2}{r^2}, \frac{x-x_0}{|x-x_0|}\big),$$ 
whose radius $\gamma r$ is scaled by the factor $\gamma=\sqrt{\frac{N+2}{\p-1}}$.
At $x=x_0$, the ellipse above coincides with the ball $B(x_0,
\gamma r)$, whereas as $|x-x_0|\to r-$ it degenerates to the
$(N-1)$-dimensional ball:
$$E\big(x,\gamma r; 0, \frac{x-x_0}{|x-x_0|}\big)= x+ \Big\{ y\in\R^N;~
\langle y, x-x_0\rangle=0 \mbox{ and } |y|< r\sqrt{\frac{N+2}{\p-1}}\;\Big\}.$$
The resulting mean value expansion is then:
\begin{equation}\label{exp_dege}
\begin{split}
\frac{1}{2}\Big(\inf_{x\in B(x_0, r)} +& \sup_{x\in B(x_0, r)}
\Big) \mathcal{A}\big(u; \gamma r, 1-\frac{|x-x_0|^2}{r^2},\frac{x-x_0}{|x-x_0|}\big)(x) 
\\ & \quad = u(x_0) +  \frac{r^2}{2(\p-1)} |\nabla u(x_0)|^{2-\p}  \Delta_\p u(x_0) +o(r^2).
\end{split}
\end{equation}
\end{Rem}

\begin{Rem}\label{PS}
In \cite{PS}, instead of averaging on an $N$-dimensional ellipse, the
average is taken on the $(N-2)$-dimensional sphere 
centered at $x$, with some radius $\gamma |x-x_0|$, and 
contained within the hyperplane perpendicular to $x-x_0$. The radius of the sphere thus
increases linearly from $0$ to $\gamma r$ with a factor $\gamma>0$, as $|x-x_0|$ varies from
$0$ to $r$. This corresponds to evaluating on $B(x_0,r)$ the
deterministic averages of:
$$f_u(x;x_0, r) = \fint_{\partial B^{N-1}(0,1)}u\big(x+\gamma |x-x_0| R(x)y \big)\dd y.$$
Here, $R(x)\in SO(N)$ is such that $R(x)e_N = \frac{x-x_0}{|x-x_0|}$,
and $\partial B^{N-1}(0,1)$ stands for the $(N-2)$-dimensional
sphere of unit radius, viewed as a subset of $\R^N$ contained in the subspace
$\R^{N-1}$ orthogonal to $e_N$. Note that $x\mapsto R(x)$ can be only
locally defined as a $\mathcal{C}^2$ function. However, the argument
as in the proof of Theorem \ref{th_expansion}, can still be applied to get:
\begin{equation*}
\begin{split}
f_u(x;x_0,r) & = u(x) +\frac{1}{2}\Big\langle \nabla^2 u(x)
:\gamma^2|x-x_0|^2 \fint_{\partial B^{N-1}(0,1)} \big(R(x)
y\big)^{\otimes 2}\dd y\Big\rangle +o(r^2) \\ & = u(x)
+\frac{\gamma^2}{2(N-1)}\Big( |x-x_0|^2 \Delta u(x) - \big\langle \nabla^2 u(x)
: (x-x_0)^{\otimes 2}\big\rangle \Big)+o(r^2) \\ & = u(x_0) +\langle \nabla
u(x_0), x-x_0\rangle +\frac{\gamma^2}{2(N-1)} |x-x_0|^2 \Delta u(x_0)
\\ & \qquad \quad \;\; \, + \Big(\frac{1}{2} - \frac{\gamma^2}{2(N-1)} \Big)\big\langle \nabla^2 u(x_0) :
(x-x_0)^{\otimes 2}\big\rangle +o(r^2) 
\end{split}
\end{equation*}
where we used the general formula $\fint_{\partial B^{d}(0,1)}y^{\otimes
  2}\dd y= \frac{1}{d}\fint_{\partial B^{d}(0,1)}|y|^2\dd y Id_d = \frac{1}{d}Id_d$, so that:
$$ \fint_{\partial B^{N-1}(0,1)} \big(R(x) y\big)^{\otimes 2}\dd y = \frac{1}{N-1} R(x)\big(Id_N - e_N^{\otimes 2}\big) R(x)^T =
\frac{1}{N-1} \Big(Id_N - \big(\frac{x-x_0}{|x-x_0|}\big)^{\otimes 2}\Big).$$
Calling $\bar{\bar f}_u$ the polynomial:
\begin{equation*}
\begin{split}
\bar{\bar f}_u(x; x_0,r) =  & ~u(x_0) + \langle \nabla u(x_0), x-x_0\rangle 
\\ & + \Big\langle \frac{\gamma^2}{2(N-1)} \Delta u(x_0) Id_N +
\Big(\frac{1}{2} - \frac{\gamma^2}{2(N-1)} \Big)\nabla^2 u(x_0) : (x-x_0)^{\otimes 2}\Big\rangle, 
\end{split}
\end{equation*}
the claim in Step 2 of proof of Theorem \ref{th_expansion} yields:
\begin{equation*}
\begin{split}
\frac{1}{2}\Big(&\inf_{x\in B(x_0,r)} \bar{\bar f}_u(x; x_0, r)  + \sup_{x\in
  B(x_0,r)}\bar{\bar f}_u(x; x_0, r) \Big) \\ &  = u(x_0)  + \frac{\gamma^2
  r^2}{2(N-1)}\Big( \Delta u(x_0) +  \Big(\frac{N-1}{\gamma^2} - 1
\Big)\Delta_\infty u(x_0) \Big) + O(r^3).
\end{split}
\end{equation*}
Clearly, there holds $ \frac{N-1}{\gamma^2} - 1=\p-2$, precisely for the
scaling factor $\gamma = \sqrt{\frac{N-1}{\p-1}}$  as in \cite{PS}.
In this case, we get the mean value expansion with the same coefficient as in (\ref{exp_dege}):
\begin{equation}\label{PS_dpp}
\begin{split}
\frac{1}{2}\Big(\inf_{x\in B(x_0, r)} +& \sup_{x\in B(x_0, r)} \Big) f_u(x;x_0,r)
\\ & \quad = u(x_0) +  \frac{r^2}{2(\p-1)} |\nabla u(x_0)|^{2-\p}  \Delta_\p u(x_0) +o(r^2).
\end{split}
\end{equation}
\end{Rem}

\section{The dynamic programming principle and the
  first convergence theorem}\label{sec_dpp_analysis}

Let $\mathcal{D}\subset\R^N$ be an open, bounded,
connected domain and let $F\in\mathcal{C}(\R^N)$ be a bounded data
function. Given $\gamma_\p, a_\p>0$ as in (\ref{prope}), recall the
definition of $S_r$ in (\ref{aveS}). We then have:

\begin{Teo}\label{thD}
For every $\epsilon\in (0,1)$ there exists a unique Borel, bounded function
$u:\R^N\to\mathbb{R}$ (denoted further by $u_\epsilon$), automatically continuous, such that:
\begin{equation}\label{DPP2} 
u(x) = d_\epsilon(x) S_\epsilon u(x) + \big(1-d_\epsilon(x)\big)F(x)
\qquad \mbox{ for all }\;x\in\R^N,
\end{equation}
where the scaled distance function $d_\epsilon:\R^N\to [0,1]$ is given by:
$$d_\epsilon (x) = \frac{1}{\epsilon}\min\big\{\epsilon,
\mathrm{dist}\big(x, \R^N\setminus\D\big)\big\}.$$
The solution operator to (\ref{DPP2}) is monotone, i.e. if $F\leq \bar
F$ then the corresponding solutions satisfy: $u_\epsilon\leq \bar
u_\epsilon$. Moreover $\|u\|_{\mathcal{C}(\R^N)}\leq \|F\|_{\mathcal{C}(\R^N)}$.
\end{Teo}
\begin{proof}
{\bf 1.} The solution $u$ of (\ref{DPP2}) is a fixed point of the
operator $T_\epsilon v = d_\epsilon S_\epsilon v + (1-d_\epsilon)F$. Recall that:
\begin{equation}\label{average}
\begin{split}
& (S_\epsilon v)(x)  = \frac{1}{2}\Big(\inf_{z\in B(0,1)} + \sup_{z\in B(0,1)} \Big)
f_v(x+\epsilon z; x, \epsilon) \\ & \mbox{where: }\quad 
f_v(x+\epsilon z; x,\epsilon) = \fint_{x+\epsilon E(z,\gamma_\p; 1+(a_\p-1)|z|^2, \frac{z}{|z|})}v(w)\dd w.
\end{split}
\end{equation}
Observe that for a fixed $\epsilon$ and $x$, and given a bounded Borel
function $v:\R^N\to \R$, the average $f_v$ is continuous in $z\in\bar B(0,1)$. In view of  continuity of the weight
$d_\epsilon$ and the data $F$, it is not hard to conclude that both
$T_\epsilon,S_\epsilon $ likewise return a bounded continuous function, so in particular
the solution to (\ref{DPP2}) is automatically continuous.
We further note that  $S_\epsilon $ and $T_\epsilon$ are monotone, namely:
$S_\epsilon v\leq S_\epsilon \bar v$ and $T_\epsilon v\leq T_\epsilon\bar v$
if $v\leq \bar v$. 

The solution $u$ of (\ref{DPP2}) is obtained as the limit of
iterations $u_{n+1}=T_\epsilon u_n$, where we set $u_0\equiv const \leq \inf
F$. Since $u_1=T_\epsilon u_0 \geq u_0$ on $\R^N$, by monotonicity of $T_\epsilon$, the
sequence $\{u_n\}_{n=0}^\infty$ is nondecreasing.  It is also bounded
(by $\|F\|_{\mathcal{C}(\R^N)}$) and thus it converges pointwise to a
(bounded, Borel) limit $u:\R^N\to\mathbb{R}$. Observe now that:
\begin{equation}\label{Sfact}
\begin{split}
|T_\epsilon u_n (x) & - T_\epsilon u(x)| \leq |S_\epsilon u_n(x)- S_\epsilon u(x)|  \\ & \leq  \sup_{z\in B(0,1)}
\fint_{x+\epsilon E(z,\gamma_\p; 1+(a_\p-1)|z|^2, \frac{z}{|z|})}
|u_n-u|(w)\dd w \leq C_\epsilon \int_{\mathcal{D}} |u_n - u|(w)\dd w,
\end{split}
\end{equation}
where $C_\epsilon$ is the lower bound on the volume of the sampling
ellipses. By the monotone convergence theorem, it follows that the
right hand side in (\ref{Sfact}) converges to $0$ as
$n\to\infty$. Consequently, $u=T_\epsilon u$, proving existence of solutions to (\ref{DPP2}).

\medskip

{\bf 2.} We now show uniqueness. If  $u, \bar u$ both solve
(\ref{DPP2}), then define $M=\sup_{x\in\R^N}|u(x)-\bar u(x)| = \sup_{x\in\mathcal{D}}|u(x)-\bar
u(x)|$ and consider any maximizer $x_0\in\mathcal{D}$, where $|u(x_0)-\bar u(x_0)| =M$.
By the same bound in (\ref{Sfact}) it follows that:
\begin{equation*}
M  =|u(x_0)-\bar u(x_0)| = d_\epsilon(x_0) |S_\epsilon u(x_0)- S_\epsilon \bar u(x_0)|
\leq \sup_{z\in B(0,1)} f_{|u-\bar u|}(x+\epsilon z; x,\epsilon)\leq M,
\end{equation*}
yielding in particular $\fint_{B(x_0, \gamma_\p\epsilon)}|u-\bar
u|(w)\dd w=M$. Consequently,
$B(x_0, \gamma_\p\epsilon) \subset D_M= \{|u-\bar u|=M\}$ and hence 
the set $D_M$ is open in $\R^N$. Since $D_M$ is 
obviously closed and nonempty, there must be $D_M=\R^N$ and since
$u-\bar u=0$ on $\R^N\setminus \mathcal{D}$,  it follows that $M=0$.
Thus $u=\bar u$, proving the claim.
Finally, we remark that the monotonicity of $S_\epsilon$ yields the monotonicity of the solution
operator to (\ref{DPP2}).
\end{proof}

\begin{Rem}
It follows from (\ref{Sfact}) that the sequence
$\{u_n\}_{n=1}^\infty$ in the proof of Theorem \ref{thD} converges to
$u=u_\epsilon$ uniformly. In fact, the iteration procedure $u_{n+1}=Tu_n$
started by any bounded and continuous function $u_0$ converges uniformly
to the uniquely given $u_\epsilon$. We further remark that if $F$ is Lipschitz
continuous then $u_\epsilon$ is likewise Lipschitz, with Lipschitz
constant depending (in nondecreasing manner) on the following quantities: $1/\epsilon$,
$\|F\|_{\mathcal{C}(\partial\mathcal{D})}$ and the Lipschitz constant
of $F_{\mid\partial\mathcal{D}}$.
\end{Rem}

\begin{Rem}
If we replace $d_\epsilon$ in Theorem \ref{thD} by the indicator
function $\chi_\D$, the resulting solutions to (\ref{DPP2}) will in
general be discontinuous, regardless of the regularity of $F$. The
classical Ascoli-Arzel\`a theorem could not be thus used in this
setting for the proof of
convegence of $\{u_\epsilon\}_{\epsilon\to 0}$. It would still be possible,
however, to obtain the asymptotic regularity and prove the uniform
convergence (see section \ref{sec_convp}) by analyzing the dependence
of variation of $u_\epsilon$ on $\epsilon$.
Another possible interpolation weight in (\ref{DPP2}) is: $\tilde d_\epsilon (x) = \frac{1}{\epsilon}\min\big\{\epsilon,
\mathrm{dist}\big(x, (\R^N\setminus\D)+\bar B_\epsilon(0)\big)\big\}$,
which varies from $0$ to $1$ on the
$(\epsilon, 2\epsilon)$ boundary layer, rather than the layer
$(0,\epsilon)$. All results in this work remain valid with $\tilde
d_\epsilon$, whereas the advantage of such a choice is that the resulting game stopping
position always takes place in $\D$.
\end{Rem}

\bigskip

We prove the following first convergence result. Our argument will be 
analytical, although a probabilistic proof is possible as well,
based on the interpretation of $u_\epsilon$ in Theorem \ref{are_equal}.

\begin{Teo}\label{conv_nondegene}
Let $F\in\mathcal{C}^2(\R^N)$ be a bounded data function that satisfies on some
open set $U$,  compactly containing $\mathcal{D}$: 
\begin{equation}\label{assu1}
\Delta_{\p}F = 0\quad \mbox{ and } \quad \nabla F\neq 0 \qquad \mbox{in }\;  U.
\end{equation}
Then the solutions $u_\epsilon$ of (\ref{DPP2}) converge to $F$ uniformly in $\R^N$, namely:
\begin{equation}\label{unicon1}
\|u_\epsilon - F\|_{\mathcal{C}(\mathcal{D})}\leq C\epsilon \qquad \mbox{ as }\; \epsilon\to 0,
\end{equation}
with a constant $C$ depending on $F$, $U$, $\mathcal{D}$ and $\p$, but
not on $\epsilon$.
\end{Teo}
\begin{proof}
{\bf 1.} We first note that since $u_\epsilon=F$ on $\R^N\setminus
\mathcal{D}$ by construction, (\ref{unicon1}) indeed implies the uniform
convergence of $u_\epsilon$ in $\R^N$. 
Also, by translating $\mathcal{D}$ if necessary, we may without loss
of generality assume that $B(0,1)\cap U=\emptyset$.

We now show that there exists $s\geq 2$ and $\hat{\epsilon}>0$ such that the following functions:
\begin{equation*}
v_\epsilon(x)=F(x)+\epsilon|x|^s
\end{equation*}
satisfy, for every $\epsilon\in (0, \hat\epsilon)$:
\begin{equation}\label{zzz}
\nabla v_\epsilon \neq 0 \quad \mbox{ and }\quad \Delta_{\p} v_\epsilon \geq
\epsilon s \cdot |\nabla {v_\epsilon}|^{{\p}-2} \quad \mbox{ in }\;\bar{\mathcal{D}}.
\end{equation}
Observe first the following direct formulas:
\begin{equation*}
\begin{split}
& \nabla |x|^s = s|x|^{s-2}x, \qquad \nabla^2|x|^s = s(s-2)
|x|^{s-4}x^{\otimes 2} + s |x|^{s-2} Id_N, \\ &
\Delta |x|^s = s(s-2+N)|x|^{s-2}.
\end{split}
\end{equation*}
Fix $x\in \bar\D$ and denote $a=\nabla v_\epsilon(x)$ and $b=\nabla F(x)$. Then, by (\ref{assu1}) we have:
\begin{equation}\label{jedena}
\begin{split}
\Delta_\p v_\epsilon(x) & = |\nabla v_\epsilon(x)|^{\p-2}
\Bigg(\epsilon \Delta |x|^s + (\p-2) \epsilon \Big\langle  \nabla^2|x|^s :
\big(\frac{a}{|a|}\big)^{\otimes 2} \Big\rangle \\ &
\qquad\qquad\qquad\qquad\qquad  +  (\p-2) \Big\langle \nabla^2F(x) :
\big(\frac{a}{|a|}\big)^{\otimes 2}- \big(\frac{b}{|b|}\big)^{\otimes 2} \Big\rangle\Bigg) \\ &
\geq |\nabla v_\epsilon(x)|^{\p-2} \epsilon s |x|^{s-2} \Bigg( s-2+N + (\p-2)
\Big(1 - \frac{4|\nabla^2F(x)|}{|\nabla F(x)|} |x|\Big)\Bigg).
 \end{split}
\end{equation}
Above, we have also used the bound:
$$ \Big\langle \nabla^2|x|^s : \big(\frac{a}{|a|} \big)^{\otimes 2} \Big\rangle = s
(s-2) |x|^{s-2} \Big\langle \frac{a}{|a|}, \frac{x}{|x|} \Big\rangle^2 + s|x|^{s-2} \geq s |x|^{s-2},$$
together with the straightforward estimate:
$ \big|\big(\frac{a}{|a|}\big)^{\otimes 2} -
\big(\frac{b}{|b|}\big)^{\otimes 2}\big|\leq 4\frac{|a-b|}{|b|}. $
The claim (\ref{zzz}) hence follows by fixing a large  exponent $s$ that satisfies:
\begin{equation}\label{s_bound}
{s\geq 3-N+ |\p-2|\cdot \max_{y\in\bar\D} \Big\{4|y|\frac{|\nabla^2 F(y)|}{|\nabla F(y)|}}\Big\},
\end{equation}
so that the quantity in the last
line parentheses in (\ref{jedena}) is greater than $1$, and further
taking $\epsilon > 0$ small enough to have: $\min_{\bar\D} |\nabla v_\epsilon|>0$.

\medskip

{\bf 2.} We claim that $s$ and $\hat{\epsilon}$ in step 1 can further
be chosen in a way that for all $\epsilon\in(0,\hat{\epsilon})$:
\begin{equation}\label{want}
v_\epsilon\leq S_\epsilon v_\epsilon \quad \text{in}\,\,\bar{\mathcal{D}}.
\end{equation}
Indeed, a careful analysis of the remainder terms in the expansion \eqref{expansion} reveals that:  
\begin{equation}\label{want2}
v_\epsilon(x) - S_\epsilon v_\epsilon(x)= - \frac{\epsilon^2}{\p-1}
|\nabla{v}_\epsilon(x)|^{2-\p}\Delta_{\p} v_\epsilon(x) + R_2(\epsilon,s), 
\end{equation} 
where: 
$$|R_2(\epsilon,s)|\leq C_\p \epsilon^2
\displaystyle{\osc_{B(x,(1+\gamma_p)\epsilon)}}|\nabla^{2} v_\epsilon|+C \epsilon^3.$$ 
Above, we denoted by $C_\p$ a constant depending only on $\p$, whereas
$C$ is a constant depending only $|\nabla {v_\epsilon}|$
and $|\nabla^2v_\epsilon|$, that
remain uniformly bounded  for small $\epsilon$.
Since $v_\epsilon$ is the sum of the smooth on $U$ function $x\mapsto \epsilon|x|^s$,
and a $\p$-harmonic function $F$ that is also smooth in virtue of its non vanishing
gradient  (this is a classical result \cite{Morrey_book}),
we obtain that  (\ref{want2}) and (\ref{zzz}) imply (\ref{want}) for $s$ sufficiently large and taking
$\epsilon$ appropriately small.

\medskip

{\bf 3.} Let $A$ be a compact set in: $\mathcal{D}\subset
A\subset U$. Fix $\epsilon\in (0,\hat\epsilon)$ and for each $x\in A$ consider:
$$\phi_\epsilon(x)=v_\epsilon(x)-u_\epsilon(x) = F(x)-u_\epsilon(x)+\epsilon|x|^s.$$ 
By (\ref{want}) and (\ref{DPP2}) we get:
\begin{equation}\label{13}
\begin{split}
\phi_\epsilon(x) &= d_\epsilon(x)( v_\epsilon(x)-S_\epsilon u_\epsilon(x) ) + (1-d_\epsilon(x))(v_\epsilon(x)-F(x))\\
&\leq d_\epsilon(x)( S_\epsilon v_\epsilon(x)-S_\epsilon u_\epsilon(x) )
+( 1-d_\epsilon(x) )(v_\epsilon(x)-F(x))\\
&\leq d_\epsilon(x) \sup_{y\in  B(0,1)}f_{\phi_\epsilon}\big(x+\epsilon y, x, \epsilon\big) + (
1-d_\epsilon(x) )\big(v_\epsilon(x)-F(x)\big). 
\end{split}
\end{equation}
Define: 
$$M_\epsilon=\max_{A}\phi_\epsilon.$$ 
We claim that there exists $x_0\in A$ with $d_\epsilon(x_0)<1$ and such that
$\phi_\epsilon(x_0)=M_\epsilon$.
To prove the claim, define
$\mathcal{D}^\epsilon=\big\{x\in\mathcal{D};~ \mbox{dist}(x,\partial\mathcal{D})\geq \epsilon \}$. 
We can assume that the closed set $\mathcal{D}^\epsilon\cap\{\phi_\epsilon=M_\epsilon\}$ is nonempty; 
otherwise the claim would be obvious. Let $\mathcal{D}_0^\epsilon$ be
a nonempty connected component of $\mathcal{D}^\epsilon$ and denote
$\mathcal{D}_M^\epsilon=\mathcal{D}_0^\epsilon\cap
\{\phi_\epsilon=M_\epsilon\}$. Clearly,  $\mathcal{D}_M^\epsilon$ is
closed in $\mathcal{D}_0^\epsilon$; we now show that it is also open. Let
$x\in\mathcal{D}_M^\epsilon$. Since $d_\epsilon(x)=1$ from (\ref{13}) it follows that: 
\begin{equation*}
M_\epsilon =\phi_\epsilon(x) \leq \sup_{y\in
  B(x,\epsilon)}\mathcal{A}\Big(\phi_\epsilon; \gamma_\p\epsilon +
(a_\p-1)\frac{|y-x|^2}{\epsilon^2}, \frac{y-x}{|y-x|}\Big)(y)  \leq M_\epsilon.
\end{equation*} 
Consequently, $\phi_\epsilon\equiv M_\epsilon$ in $B(x,\gamma_\p
\epsilon)$ and thus we obtain the openness 
of $\mathcal{D}_M^\epsilon$ in $\mathcal{D}_0^\epsilon$.
In particular, $\mathcal{D}_M^\epsilon$ contains a point $\bar{x}\in\partial
\mathcal{D}^\epsilon$. Repeating the previous argument for $\bar x$
results in $\phi_\epsilon\equiv M_\epsilon$ in $B(\bar{x}, \gamma_\p\epsilon)$,
proving the claim.

\medskip

{\bf 4.} We now complete the proof of Theorem \ref{conv_nondegene} by deducing a bound on $M_\epsilon$.
If $M_\epsilon=\phi_\epsilon(x_0)$ for some $x_0\in\bar{\mathcal{D}}$
with $d_\epsilon(x_0)<1$, then (\ref{13}) yields:
$ M_\epsilon=\phi_\epsilon(x_0)
\leq d_\epsilon(x_0)M_\epsilon +(1-d_\epsilon(x_0))\big(v_\epsilon(x_0)-F(x_0)\big), $
which implies:
$$ M_\epsilon\leq v_\epsilon(x_0)-F(x_0)=\epsilon|x_0|^s .$$
On the other hand, if $M_\epsilon=\phi_\epsilon(x_0)$ for some $x_0\in
A\setminus \mathcal{D}$, then $d_\epsilon(x_0)=0$, hence likewise:
$ M_\epsilon=\phi_\epsilon(x_0)= v_\epsilon(x_0)-F(x_0)=
\epsilon|x_0|^s.$ In either case:
$$ \max_{\bar{\mathcal{D}}}( u-u_\epsilon)\leq
\max_{\bar{\mathcal{D}}}\phi_\epsilon + C\epsilon \leq 2C\epsilon $$
where $C=\max_{x\in V}|x|^s$ is independent of $\epsilon$. A
symmetric argument applied to $-u$ after noting that
$(-u)_\epsilon=-u_\epsilon$ gives:  $  \min_{\bar{\mathcal{D}}}( u-u_\epsilon)\geq -2C\epsilon$.
The proof is done.
\end{proof}

\section{The random Tug of War game modelled on (\ref{expansion})}\label{sec_setup_p}

We now develop the probability setting related to the expansion (\ref{expansion}).

\bigskip

{\bf 1.} Let $\Omega_1 = B(0,1)\times \{1,2\}\times (0,1)$ and define:
\begin{equation*}
\begin{split}
\Omega= (\Omega_1)^{\mathbb{N}} = \big\{&\omega=\{(w_i, s_i,
t_i)\}_{i=1}^\infty; ~ w_i\in B(0,1),~ s_i\in \{1,2\},~ t_i\in (0,1) ~~\mbox{ for all }
i\in\mathbb{N}\big\}.
\end{split}
\end{equation*}
The probability space $(\Omega, \mathcal{F}, \mathbb{P})$ 
is given as the countable product of $(\Omega_1, \mathcal{F}_1,
\mathbb{P}_1)$. Here, $\mathcal{F}_1$ is the smallest $\sigma$-algebra
containing all products $D\times S\times B$ where $D\subset
B(0,1)\subset\mathbb{R}^N$ and $B\subset (0,1)$ are Borel, and
$S\subset\{1,2\}$. The measure $\mathbb{P}_1$ is 
the product of: the normalised Lebesgue measure on $B(0,1)$, the
uniform counting measure on $\{1,2\}$ and the Lebesgue measure on $(0,1)$:
$$\mathbb{P}_1(D\times S\times B) = \frac{|D|}{|B(0,1)|}\cdot \frac{|S|}{2}\cdot |B|.$$
For each $n\in\mathbb{N}$, the probability space $(\Omega_n, \mathcal{F}_n, \mathbb{P}_n)$ is the
product of $n$ copies of $(\Omega_1, \mathcal{F}_1,
\mathbb{P}_1)$. The $\sigma$-algebra $\mathcal{F}_n$
is always identified with the sub-$\sigma$-algebra of $\mathcal{F}$,
consisting of sets $A\times \prod_{i=n+1}^\infty\Omega_1$
for all $A\in\mathcal{F}_n$. The sequence $\{\mathcal{F}_n\}_{n=0}^\infty$ where 
$\mathcal{F}_0= \{\emptyset, \Omega\}$, is a filtration of $\mathcal{F}$. 

\medskip

{\bf 2.} Given are two families of functions $\sigma_I=\{\sigma_I^n\}_{n=0}^\infty$ and
$\sigma_{II}=\{\sigma_{II}^n\}_{n=0}^\infty$, defined on the
corresponding spaces of ``finite histories'' $H_n=\R^N\times (\R^N\times\Omega_1)^n$:
$$\sigma_I^n, \sigma_{II}^n:H_n\to B(0,1)\subset\R^N,$$
assumed to be measurable with respect to the (target) Borel
$\sigma$-algebra in $B(0,1)$ and the (domain) product $\sigma$-algebra on $H_n$.
For every $x_0\in\R^N$ and $\epsilon\in (0,1)$ we recursively define:
$$\big\{X_n^{\epsilon, x_0, \sigma_I, \sigma_{II}}:\Omega\to\R^N\big\}_{n=0}^\infty.$$
For simplicity of notation, we often suppress some of the superscripts $\epsilon, x_0, \sigma_I, \sigma_{II}$
and write $X_n$ (or $X_n^{x_0}$, or $X_n^{\sigma_I, \sigma_{II}}$,
etc) instead of $ X_n^{\epsilon, x_0, \sigma_I, \sigma_{II}}$, if no ambiguity arises. Let: 
\begin{equation}\label{processMp}
\begin{split}
& \, X_0\equiv x_0, \\ & \, X_n\big((w_1,s_1,t_1), \ldots,
(w_n,s_n,t_n)\big) \\ & \quad = x_{n-1} + \left\{\begin{array}{ll} 
\epsilon\Big(\sigma_I^{n-1}(h_{n-1}) + \gamma_\p w_n +
\gamma_\p(a_\p-1)\langle w_n, \sigma_I^{n-1}(h_{n-1})\rangle \sigma_I^{n-1}(h_{n-1})\Big) &
\mbox{for } s_n=1 \vspace{1mm} \\ \epsilon\Big(\sigma_{II}^{n-1}(h_{n-1}) + \gamma_\p w_n +
\gamma_\p(a_\p-1)\langle w_n, \sigma_{II}^{n-1}(h_{n-1})\rangle \sigma_{II}^{n-1}(h_{n-1})\Big) &
\mbox{for } s_n=2, \end{array} \right.\\
& \mbox{ where } ~~ x_{n-1}= X_{n-1}\big((w_1,s_1,t_1), \ldots,(w_{n-1},s_{n-1},t_{n-1})\big)  \\ &
\mbox{ and } ~~ ~ h_{n-1}= \big(x_0, (x_1,w_1,s_1,t_1), \ldots,
(x_{n-1},w_{n-1},s_{n-1},t_{n-1})\big)\in H_{n-1}.
\end{split}
\end{equation}
In this ``game'', the position $x_{n-1}$ is first advanced
(deterministically) according to the two players' ``strategies''
$\sigma_I$ and $\sigma_{II}$ by a shift 
$\epsilon y\in B(0, {\epsilon})$, and then (randomly) uniformly by a further
shift in the ellipse $\epsilon E\big(0,\gamma_\p; 1+(a_\p-1)|y|^2,\frac{y}{|y|}\big)$. The
deterministic shifts are activated  by the value of the equally probable outcomes:
$s_n=1$ activates $\sigma_I$ and $s_n=2$ activates $\sigma_{II}$.

\medskip

{\bf 3.} The auxiliary variables $t_n\in (0,1)$ serve as a threshold
for reading the eventual value from the prescribed boundary data.
Let $\mathcal{D}\subset\R^N$ be an open, bounded and connected
set. Define the random variable $\tau^{\epsilon, x_0, \sigma_I,
\sigma_{II}}:\Omega\to\mathbb{N}\cup\{\infty\}$ in: 
$$\tau^{\epsilon, x_0, \sigma_I,
  \sigma_{II}}\big((w_1,s_1,t_1),(w_2,s_2,t_2),\ldots\big)=\min\big\{n\geq
1;~ t_n>d_\epsilon(x_{n-1})\big\},$$
where:
$$d_\epsilon(x) = \frac{1}{\epsilon}\min\big\{\epsilon, \mbox{dist}(x,
\R^N\setminus\mathcal{D})\big\}$$ 
is the scaled distance from the complement of
$\mathcal{D}$. As before, we drop the superscripts and write $\tau$
instead of $\tau^{\epsilon, x_0, \sigma_I, \sigma_{II}}$ if there is
no ambiguity. Our ``game'' is thus terminated, with
probability $1- d_\epsilon(x_{n-1})$, whenever the
position $x_{n-1}$ reaches the $\epsilon$-neighbourhood of
$\partial\mathcal{D}$. 

\medskip

\begin{Lemma}\label{facom}
If the scaling factors $a_\p, \gamma_\p>0$ in (\ref{prope}) satisfy:
\begin{equation}\label{prope2}
a_\p\leq 1 \; \mbox{ and } \;\gamma_\p a_\p>1 \qquad \mbox{ or }\qquad 
a_\p\geq 1\; \mbox{ and } \; \gamma_\p >1, 
\end{equation}
then $\tau$ is a stopping time relative to the filtration
$\{\mathcal{F}_n\}_{n=0}^\infty$, namely: $\mathbb{P}(\tau<\infty)=1$. 
Further, for any $\p\in (1,\infty)$ there exist positive $a_\p,\gamma_\p$ with
(\ref{prope}) and (\ref{prope2}).
\end{Lemma}
\begin{proof}
Let $a_\p\leq 1$ and $\gamma_\p a_\p>1$. Then, for some $\beta>0$, there
also holds: $\gamma_\p(a_\p - \beta)>1$. Define an open set of
``advancing random shifts'':
$$D_{adv}=\big\{ w\in B(0,1);~ \langle w, e_1\rangle > 1-\beta\big\}.$$
For every $\sigma\in B(0,1)$ and every $w\in D_{adv}$ we have:
$$\gamma_\p\big\langle w+(a_\p - 1)\langle w, \sigma\rangle \sigma,
e_1\big\rangle \geq \gamma_\p \big(\langle w, e_1\rangle + a_\p - 1\big) >
\gamma_\p (a_\p-\beta).$$
Since $\D$ is bounded, the above estimate implies existence of $n\geq 1$
(depending on $\epsilon$) such that for all initial points $x_0\in\D$
and all deterministic shifts $\{\sigma^i\in B(0,1)\}_{i=1}^n$ there holds:
$$x_0 + \epsilon \sum_{i=1}^n \big(\sigma^i + \gamma_\p w_i +
\gamma_\p (a_\p-1) \langle w_i,\sigma^i\rangle
\sigma^i\big)\not\in\D\qquad \mbox{ for all }\;\{w_i\in D_{adv}\}_{i=1}^n.$$
In conclusion:
$$\mathbb{P}(\tau\leq n) \geq \mathbb{P}_n\Big( \big(D_{adv}\times \{1,2\}\times
(0,1)\big)^n\Big) = \Big(\frac{|D_{adv}|}{|B(0,1)|}\Big)^n = \eta>0$$
and so $\mathbb{P}(\tau>kn)\leq (1-\eta)^k$ for all $k\in\mathbb{N}$,
yielding: $\mathbb{P}(\tau=\infty) = \lim_{k\to\infty}\mathbb{P}(\tau>kn)=0$.

\smallskip

The proof proceeds similarly when $a_\p\geq 1$ and $\gamma_\p >1$. Fix
$\bar\beta>0$ such that $\gamma_\p(1-\bar\beta)>1$ and define
$D_{adv}$ as before, for an appropriately small $0<\beta\ll
\bar\beta$, ensuring that:
$$\gamma_\p\big\langle w+(a_\p - 1)\langle w, \sigma\rangle \sigma,
e_1\big\rangle \geq \gamma_\p \big(\langle w, e_1\rangle - (a_\p - 1) \sqrt{2\beta}\big) >
\gamma_\p (1-\bar\beta)$$
for every $\sigma\in B(0,1)$ and every $w\in D_{adv}$. Again, after at most $\Big\lceil
\frac{\mbox{diam}\;\D}{\epsilon(\gamma_\p (1-\bar\beta) -
 1)}\Big\rceil$ shifts, the token will leave the domain $\D$
(unless it is stopped earlier) and the game will be terminated.

\smallskip

It remains to prove existence of $\gamma_\p, a_p>0$ satisfying
(\ref{prope}) and (\ref{prope2}). We observe that the viability of
$a_\p\leq 1$, $\gamma_\p a_\p>1$ is equivalent to:
$\frac{1}{\gamma_\p^2}<\p-1-\frac{N+2}{\gamma_\p^2}\leq 1$ and further
to:  $\frac{\p-2}{N+2}\leq \frac{1}{\gamma_\p^2}<\frac{\p-1}{N+3}$,
which allows for choosing $\gamma_\p$ (and $a_\p$) for $\p<N+4$. On
the other hand, viability of $a_\p\geq 1$, $\gamma_\p >1$ is
equivalent to: $\gamma_\p^2>1$ and $\p-1-\frac{N+2}{\gamma_\p^2}\geq
1$, that is: $\frac{1}{\gamma_\p^2}<\min\big\{1,
\frac{\p-2}{N+2}\big\}$, yielding existence of $\gamma_\p, a_\p$ for $\p>2$.
\end{proof}

\bigskip

{\bf 4.} From now on, we will work under the additional requirement
(\ref{prope2}).  In our ``game'', the first ``player'' collects from his
opponent the payoff given by the data $F$ at the stopping position. The incentive
of the collecting ``player'' to maximize the outcome and of the
disbursing ``player'' to minimize it, leads to
the definition of the two game values below.

Let $F:\R^N\to\mathbb{R}$ be a countinuous function. Then we have:
\begin{equation}\label{ue_def_p}
\begin{split}
& u_I^\epsilon(x) =
\sup_{\sigma_I}\inf_{\sigma_{II}}\mathbb{E}\Big[F\circ \big(X^{\epsilon, x,
\sigma_I, \sigma_{II}}\big)_{\tau^{\epsilon, x, \sigma_I,  \sigma_{II}}-1}\Big], \\
& u_{II}^\epsilon(x) =
\inf_{\sigma_{II}}\sup_{\sigma_{I}}\mathbb{E}\Big[F\circ \big(X^{\epsilon, x,
\sigma_I, \sigma_{II}}\big)_{\tau^{\epsilon, x, \sigma_I, \sigma_{II}}-1}\Big]. 
\end{split}
\end{equation}
The main result in Theorem \ref{are_equal} will show that $u_I^\epsilon = u_{II}^\epsilon
\in\mathcal{C}(\R^N)$ coincide with the unique solution to the
dynamic programming principle in section \ref{sec_dpp_analysis},
modelled on the expansion (\ref{expansion}). It is also clear that 
$u_{I, II}^\epsilon$ depend only on the values of $F$ in the
$\epsilon$-neighbourhood of $\partial\mathcal{D}$. In
section \ref{sec_convp} we will prove that as $\epsilon\to 0$, the
uniform limit of $u_{I, II}^\epsilon$ that depends only on $F_{\mid\partial\mathcal{D}}$,
is $\p$-harmonic in $\mathcal{D}$ and attains $F$ on
$\partial\mathcal{D}$, provided that $\partial\mathcal{D}$ is regular.

\bigskip

\begin{Teo}\label{are_equal}
For every $\epsilon\in (0,1)$, let $u_I^\epsilon$, $u_{II}^\epsilon$
be as in (\ref{ue_def_p}) and $u_\epsilon$ as in Theorem \ref{thD}. Then: 
$$u_I^\epsilon =  u_\epsilon =u_{II}^\epsilon.$$
\end{Teo}
\begin{proof}
{\bf 1.} We drop the sub/superscript $\epsilon$ to ease the notation.
To show that $u_{II}\leq u$, fix $x_0\in\R^N$ and
$\eta>0$. We first observe that there exists a strategy
$\sigma_{0,II}$ where $\sigma_{0,II}^n(h_n) = \sigma_{0,II}^n(x_n)$
satisfies for every $n\geq 0$ and $h_n\in H_n$:
\begin{equation}\label{infimize}
f_u(x_n + \epsilon \sigma_{0,II}^n(x_n); x_n, \epsilon)\leq \inf_{z\in
B(0,1)}f_u(x_n+\epsilon z; x_n, \epsilon)+\frac{\eta}{2^{n+1}} 
\end{equation}
Indeed, using the continuity of (\ref{f}), we note that there exists $\delta>0$ such that:
$$\big|\inf_{z\in B(0,1)}f_u(x+\epsilon z; x, \epsilon) -  \inf_{z\in
  B(0,1)}f_u(\bar x+\epsilon z; \bar x, \epsilon) \big| < \frac{\eta}{2^{n+2}} \quad \mbox{for
all }\; |x-\bar x|<\delta.$$
Let $\{B(x_i,\delta)\}_{i=1}^\infty$ be a locally finite covering of
$\R^N$. For each $i=1\ldots \infty$, choose $z_i\in B(0,1)$
satisfying: $|\inf_{z\in B(01)}f_u(x_i+\epsilon z; x_i, \epsilon) -
f_u(x_i+\epsilon z_i; x_i, \epsilon)|<\frac{\eta}{2^{n+2}}$. Finally, define:
$$\sigma^n_{0,II}(x) = z_i \quad \mbox{for }\; x\in B(x_i, \delta)\setminus
\bigcup_{j=1}^{i-1} B(x_j, \delta).$$
The piecewise constant function $\sigma^n_{0,II}$ is obviously Borel
and it satisfies (\ref{infimize}).

\medskip

{\bf 2.} Fix a strategy $\sigma_I$ and consider the following sequence of
random variables $M_n:\Omega\to\mathbb{R}$:
$$M_n=(u\circ X_n)\mathbbm{1}_{\tau>n} + (F\circ
X_{\tau-1})\mathbbm{1}_{\tau\leq n} + \frac{\eta}{2^n}.$$
We show that $\{M_n\}_{n=0}^\infty$ is a supermartingale with respect
to the filtration $\{\mathcal{F}_n\}_{n=0}^\infty$. Clearly:
\begin{equation}\label{decompo}
\begin{split}
\mathbb{E}\big(M_n\mid\mathcal{F}_{n-1}\big) = ~ &
\mathbb{E}\big((u\circ X_n)\mathbbm{1}_{\tau>n}\mid\mathcal{F}_{n-1}\big) + 
\mathbb{E}\big((F\circ X_{n-1})\mathbbm{1}_{\tau=n}\mid\mathcal{F}_{n-1}\big) \\ & + 
\mathbb{E}\big((F\circ X_{\tau -1})\mathbbm{1}_{\tau<n}\mid\mathcal{F}_{n-1}\big) +
\frac{\eta}{2^n}\qquad \mbox{a.s.} 
\end{split}
\end{equation}
We readily observe that: $\mathbb{E}\big((F\circ X_{\tau
  -1})\mathbbm{1}_{\tau<n}\mid\mathcal{F}_{n-1}\big) 
= (F\circ X_{\tau -1})\mathbbm{1}_{\tau<n}$. Further, writing
$\mathbbm{1}_{\tau=n} = \mathbbm{1}_{\tau\geq n} \mathbbm{1}_{t_n>d_\epsilon(x_{n-1})}$, it follows that:
\begin{equation*}
\begin{split}
\mathbb{E}\big((F\circ X_{n-1})&\mathbbm{1}_{\tau=n}\mid\mathcal{F}_{n-1}\big) =
\mathbb{E}\big(\mathbbm{1}_{t_n>d_\epsilon(x_{n-1})}\mid\mathcal{F}_{n-1}\big) \cdot
(F\circ X_{n-1}) \mathbbm{1}_{\tau\geq n} \\ & = \big(1- d_\epsilon(x_{n-1})\big)
(F\circ X_{n-1}) \mathbbm{1}_{\tau\geq n}  \qquad \mbox{a.s.} 
\end{split}
\end{equation*}
Similarly, since $\mathbbm{1}_{\tau>n} = \mathbbm{1}_{\tau\geq n}
\mathbbm{1}_{t_n\leq d_\epsilon(x_{n-1})}$, we get in view of (\ref{infimize}):
\begin{equation*}
\begin{split}
\mathbb{E}&\big((u\circ X_{n})\mathbbm{1}_{\tau>n}\mid\mathcal{F}_{n-1}\big) =
\mathbb{E}\big(u\circ X_n \mid\mathcal{F}_{n-1}\big)\cdot
d_\epsilon(x_{n-1}) \mathbbm{1}_{\tau\geq n} \\ & = \int_{\Omega_1} u\circ
X_n\dd \mathbb{P}_1 \cdot d_\epsilon(x_{n-1}) \mathbbm{1}_{\tau\geq n} 
\\ & = \frac{1}{2}\bigg(\mathcal{A}\Big(u;\gamma_\p\epsilon, 1+(a_\p-1)
|\sigma^{n-1}_I|^2, \frac{\sigma^{n-1}_I}{|\sigma^{n-1}_I|}\Big)(x_{n-1}+\epsilon\sigma_I^{n-1})  \\ &
\qquad \qquad + \mathcal{A}\Big(u;\gamma_\p\epsilon, 1+(a_\p-1)
|\sigma_{0,II}^{n-1}|^2, \frac{\sigma^{n-1}_{0,II}}{|\sigma^{n-1}_{0, II}|}\Big) (x_{n-1}+\epsilon\sigma_{0,II}^{n-1})\bigg)
\cdot d_\epsilon(x_{n-1}) \mathbbm{1}_{\tau\geq n}  \\ & \leq
\big( S\circ X_{n-1} +\frac{\eta}{2^n}\big) d_\epsilon(x_{n-1}) \mathbbm{1}_{\tau\geq n} 
\qquad \mbox{a.s.} 
\end{split}
\end{equation*}

Concluding, by (\ref{DPP2}) the decomposition (\ref{decompo}) yields:
\begin{equation*}
\begin{split}
\mathbb{E}\big(M_n\mid\mathcal{F}_{n-1}\big) & \leq
\Big(d_\epsilon(x_{n-1})\big(S\circ X_{n-1}\big) +
(1-d_\epsilon(x_{n-1}))\big(F\circ X_{n-1}\big)\Big)
\mathbbm{1}_{\tau\geq n} \\ & + (F\circ X_{\tau-1})
\mathbbm{1}_{\tau\leq n-1} + \frac{\eta}{2^{n-1}} = M_{n-1} \qquad \mbox{a.s.} 
\end{split}
\end{equation*}

\medskip

{\bf 3.} The supermartingale property of  $\{M_n\}_{n=0}^\infty$ being
established, we conclude that:
$$u(x_0) +\eta = \mathbb{E}\big[ M_0 \big] \geq \mathbb{E}\big[ M_\tau
\big] = \mathbb{E}\big[ F\circ X_{\tau-1}\big] +\frac{\eta}{2^\tau}. $$
Thus: 
$$u_{II}(x_0) \leq \sup_{\sigma_{I}}\mathbb{E}\big[ F\circ
(X^{\sigma_I, \sigma_{II,0}})_{\tau-1}\big] \leq u(x_0)+\eta.$$
As $\eta>0$ was arbitrary, we obtain the claimed comparison
$u_{II}(x_0)\leq u(x_0)$. For the reverse inequality $ u(x_0)\leq u_{I}(x_0)$,
we use a symmetric argument, with an almost-maximizing strategy
$\sigma_{0,I}$ and the resulting submartingale $\bar M_n=(u\circ X_n)\mathbbm{1}_{\tau>n} + (F\circ
X_{\tau-1})\mathbbm{1}_{\tau\leq n} - \frac{\eta}{2^n}$, along a given
yet arbitrary strategy $\sigma_{II}$. The obvious estimate
$u_{I}(x_0)\leq u_{II}(x_0)$ concludes the proof.
\end{proof}

\section{Convergence of $u_\epsilon$ and game-regularity}\label{sec_convp}

Towards checking convergence of the family $\{u_\epsilon\}_{\epsilon\to 0}$, we
first show that its equicontinuity is  implied by the equicontinuity ``at $\partial\mathcal{D}$''. This
last property will be, in turn, implied by the ``game-regularity''
condition, which in the context of stochastic Tug of War games has been introduced in \cite{PS}.
Below, we present an analytical proof. A probabilistic argument
could be carried out as well, based on a game translation argument.

\medskip

Let $\mathcal{D}\subset\R^N$ be an open, bounded,
connected domain and let $F\in\mathcal{C}(\R^N)$ be a bounded
data function.  We have the following:

\begin{Teo}\label{transfer_p}
Let $\{u_\epsilon\}_{\epsilon\to 0}$ be the family of solutions  to
(\ref{DPP2}). Assume that for every $\eta>0$ there exists $\delta>0$
and $\hat\epsilon\in (0,1)$ such that for all $\epsilon\in (0,\hat\epsilon)$ there holds:
\begin{equation}\label{bd_asp}
|u_\epsilon(y_0)- u_\epsilon(x_0)|\leq \eta \qquad \mbox{for all }\;
y_0\in\mathcal{D}, ~ x_0\in\partial\mathcal{D} \; \mbox{ satisfying
}\; |x_0-y_0|\leq \delta.
\end{equation}
Then the family $\{u_\epsilon\}_{\epsilon\to 0}$ is equicontinuous in $\bar{\mathcal{D}}$.
\end{Teo}
\begin{proof}
For every small $\hat\delta>0$, define the open, bounded, connected
set $\mathcal{D}^{\hat\delta}$ and the distance:
\begin{equation*}
\mathcal{D}^{\hat\delta} = \big\{q\in\mathcal{D}; ~ \mbox{dist}(q,
\mathbb{R}^N\setminus \mathcal{D})>{\hat\delta}\big\} \quad\mbox{ and
}\quad d^{\hat\delta}_\epsilon(q) = \frac{1}{\epsilon}\min\{\epsilon,
\mbox{dist}(q,\R^N\setminus \mathcal{D}^{\hat\delta})\}.
\end{equation*}
Fix $\eta>0$. In view of (\ref{bd_asp}) and since without loss
of generality the data function $F$ is constant outside of some large bounded
superset of $\mathcal{D}$ in $\R^N$, there exists $\hat\delta>0$  satisfying:
\begin{equation}\label{bd_asp2}
|u_\epsilon(x+z)- u_\epsilon(x)|\leq \eta \qquad \mbox{for all }\;
x\in\R^N\setminus \mathcal{D}^{\hat\delta}, ~~|w|\leq\hat\delta, ~~
\epsilon\in (0,\hat\epsilon).
\end{equation}
Fix $x_0, y_0\in \bar{\mathcal{D}}$ with $|x_0-y_0|\leq\frac{\hat\delta}{2}$ 
and let $\epsilon\in (0, \frac{\hat\delta}{2})$. Consider the
following function $\tilde u_\epsilon\in\mathcal{C}(\R^N)$: 
$$\tilde u_\epsilon(x) = u_\epsilon(x-(x_0-y_0)) + \eta.$$
Then, by (\ref{DPP2}) and recalling the definition of the principal
averaging operator $S_\epsilon$, we get:
\begin{equation}\label{bd_bd}
\begin{split}
(S_\epsilon \tilde u_\epsilon)(x) = (S_\epsilon u_\epsilon)(x-(x_0-y_0))+\eta =
u_\epsilon(x-(x_0-y_0)) + \eta = \tilde u_\epsilon(x) \quad
\mbox{for all }\; x\in\mathcal{D}^{\hat\delta}.
\end{split}
\end{equation}
because in $\mathcal{D}^{\hat\delta}$ there holds:
$$ \mbox{dist}(x-(x_0-y_0), \R^N\setminus \mathcal{D})\geq \mbox{dist}(x,
\R^N\setminus \mathcal{D}) - |x_0-y_0|\geq 
{\hat\delta}- \frac{\hat\delta}{2}= \frac{\hat\delta}{2}>\epsilon.$$

It follows now from (\ref{bd_bd}) that:
\begin{equation*}
\begin{split}
\tilde u_\epsilon= d_\epsilon^{\hat\delta}  (S_\epsilon \tilde u_\epsilon)
+ \big(1 -d_\epsilon^{\hat\delta}\big) \tilde u_\epsilon\qquad  \mbox{in }\; \R^N.
\end{split}
\end{equation*}
On the other hand, $u_\epsilon$ itself similarly solves the same
problem above, subject to its own data $u_\epsilon$ on $\R^N\setminus
\mathcal{D}^{\hat\delta}$. Since for every $x\in\R^N\setminus
\mathcal{D}^{\hat\delta}$ we have: $\tilde u_\epsilon(x) - u_\epsilon(x)
= u_\epsilon(x-(x_0-y_0))-u_\epsilon(x) + \eta \geq 0$ in view of
(\ref{bd_asp2}), the monotonicity property in
Theorem \ref{thD} yields:
$$u_\epsilon\leq \tilde u_\epsilon \qquad\mbox{in }\; \R^N.$$
Thus, in particular: $u_\epsilon(x_0)-u_\epsilon(y_0)\leq
\eta$. Exchanging $x_0$ with $y_0$ we get the opposite inequality,
and hence $|u_\epsilon(x_0)-u_\epsilon(y_0)|\leq\eta$,
establishing the claimed equicontinuity of
$\{u_\epsilon\}_{\epsilon\to 0}$ in $\bar{\mathcal{D}}$.
\end{proof}

\medskip

Following \cite{PS}, we say that a point $x_0\in\partial\mathcal{D}$ is
{game-regular} if, whenever the game starts near $x_0$,
one of the ``players'' has a strategy for making the game terminate
still near $x_0$, with high probability. More precisely:

\begin{Def}\label{def_gamereg}
Consider the Tug of War game with noise in (\ref{processMp}) and (\ref{ue_def_p}).
\begin{itemize}
\item[(a)] We say that a point $x_0\in\partial\mathcal{D}$ is {\em game-regular} if for
every $\eta, \delta>0$ there exist $\hat\delta\in (0, \delta)$ and $\hat\epsilon\in
(0, 1)$ such that the following holds. Fix
$\epsilon\in (0, \hat\epsilon)$ and $x\in B(x_0, {\hat\delta})$; there exists then a strategy
$\sigma_{0, I}$ with the property that for every strategy $\sigma_{II}$ we have:
\begin{equation}\label{gar}
\mathbb{P}\big( (X^{\epsilon, x, \sigma_{0,I}, \sigma_{II}})_{\tau-1}\in B(x_0, \delta )\big) \geq 1-\eta.
\end{equation}

\item[(b)] We say that  {$\mathcal{D}$ is game-regular} if every
  boundary point $x_0\in\partial\mathcal{D}$ is game-regular.
\end{itemize}
\end{Def}

\begin{Rem}\label{usp}
If condition (b) holds, then $\hat\delta$ and $\hat\epsilon$ in part (a) can be chosen
independently of $x_0$. Also, game-regularity is symmetric with
respect to $\sigma_I$ and $\sigma_{II}$.
\end{Rem}

\medskip

\begin{Lemma}\label{thm_gamenotreg_noconv}
Assume that for every bounded data $F\in\mathcal{C}(\R^N)$, the family of solutions
$\{u_\epsilon\}_{\epsilon\to 0}$ of (\ref{DPP2})  is equicontinuous in
$\bar{\mathcal{D}}$. Then $\mathcal{D}$ is game-regular. 
\end{Lemma}
\begin{proof}
Fix $x_0\in\partial\mathcal{D}$ and let $\eta, \delta\in (0,1)$. Define the
data function:  $F(x) = -\min\big\{1, |x-x_0|\big\}.$
By assumption and since $u_\epsilon(x_0)=F(x_0)=0$, there exists $\hat\delta\in
(0,\delta)$ and $\hat\epsilon\in (0,1)$ such that:
$$ |u_\epsilon (x)|< \eta\delta \;\;\;\mbox{ for all } \; x\in
B(x_0,{\hat\delta})  \;\mbox{ and } \; \epsilon\in (0,\hat\epsilon).$$
Consequently:
$$\sup_{\sigma_I}\inf_{\sigma_{II}}\mathbb{E}\big[F\circ (X^{\epsilon, x})_{\tau-1}\big]
= u_I^\epsilon(x) > -\eta\delta,$$
and thus there exists  $\sigma_{0,I}$ with the property
that: $\mathbb{E}\big[F\circ (X^{\epsilon, x, \sigma_{0,I},
  \sigma_{II}})_{\tau-1}\big]>-\eta\delta$ for every strategy $\sigma_{II}$. Then:
$$\mathbb{P}\big(X_{\tau-1}\not\in B(x_0, \delta)\big) \leq
-\frac{1}{\delta}\int_\Omega F(X_{\tau-1})\dd\mathbb{P}<\eta,$$
proving (\ref{gar}) and hence game-regularity of $x_0$.
\end{proof}

\begin{Teo}\label{thm_gamereg_conv}
Assume that $\mathcal{D}$ is game-regular. Then, for every bounded data
$F\in\mathcal{C}(\R^N)$, the family $\{u_\epsilon\}_{\epsilon\to 0}$ of solutions to (\ref{DPP2})
is equicontinuous in  $\bar{\mathcal{D}}$.
\end{Teo}
\begin{proof}
In virtue of Theorem \ref{transfer_p} it is enough to validate the condition
(\ref{bd_asp}). To this end, fix $\eta>0$ and let $\delta>0$ be such that:
\begin{equation}\label{ga}
|F(x)-F(x_0)|\leq \frac{\eta}{3} \;\;\mbox{ for all }
\;\;x_0\in\partial\mathcal{D} \; \mbox{ and } \; x\in B(x_0,\delta).
\end{equation}
By Remark \ref{usp} and Definition \ref{def_gamereg}, we may choose
$\hat\delta\in (0,\delta)$ and $\hat\epsilon\in (0, \delta)$ such that
for every $\epsilon\in (0, \hat\epsilon)$, every
$x_0\in\partial\mathcal{D}$ and every $x\in B(x_0,{\hat\delta})$,
there exists a strategy $\sigma_{0,II}$ with the property that for
every $\sigma_I$ there holds:
\begin{equation}\label{ga1}
\mathbb{P}\big((X^{\epsilon, x, \sigma_I,
\sigma_{0,II}})_{\tau-1}\in B(x_0,\delta )\big) \geq 1 - \frac{\eta}{6\|F\|_{\mathcal{C}(\R^N)+1}}.
\end{equation}
Let $x_0\in\partial\mathcal{D}$ and $y_0\in\mathcal{D}$ satisfy
$|x_0-y_0|\leq\hat\delta$. Then:  
\begin{equation*}
\begin{split}
u_\epsilon(y_0) - u_\epsilon(x_0) & =u_{II}^\epsilon(y_0)-F(x_0)\leq\sup_{\sigma_I}\mathbb{E}
\big[F\circ (X^{\epsilon, y_0, \sigma_I, \sigma_{0,II}})_{\tau-1} -
F(x_0)\big]\\ & \leq \mathbb{E}\big[F\circ (X^{\epsilon, y_0, \sigma_{0,I},
\sigma_{0,II}})_{\tau-1} - F(x_0)\big] + \frac{\eta}{3},
\end{split}
\end{equation*}
for some almost-supremizing strategy $\sigma_{0,I}$. Thus, by (\ref{ga}) and (\ref{ga1}): 
\begin{equation*}
\begin{split}
u_\epsilon(y_0) - u_\epsilon(x_0) & \leq \int_{\{X_{\tau-1} \in
B(x_0,\delta)\}} |F(X_{\tau-1}) - F(x_0)|\dd\mathbb{P} \\ & 
\qquad\qquad + \int_{\{X_{\tau-1} \not\in
B(x_0,\delta)\}} |F(X_{\tau-1}) - F(x_0)|\dd\mathbb{P} +
\frac{\eta}{3}\\ & \leq \frac{\eta}{3} +
2\|F\|_{\mathcal{C}(\R^N)}\mathbb{P}\big(X_{\tau-1}\not\in B(x_0, \delta)\big)+\frac{\eta}{3}\leq\eta.
\end{split}
\end{equation*}
The remaining inequality $u_\epsilon(y_0) - u_\epsilon(x_0)>-\eta$ is
obtained by a reverse argument.
\end{proof}

\begin{Rem}
We expect the family $\{u_\epsilon\}_{\epsilon\to 0}$ always to converge pointwise
(regardless of the regularity of $\D$), to the limit function $u$
that coincides with Perron's solution of the boundary value problem:
$\Delta_pu=0$ in $\D$, $u=F$ on $\partial\D$. Condition (\ref{gar}),
implying the uniform convergence 
and the resulting attainment of the boundary values $F$ by $u$, is expected to be
equivalent with the Wiener regularity criterion \cite{4}. These
assertions may be proved directly in the harmonic case $\p=2$, and
they will be the subject of future work in the nonlinear setting  $\p\neq 2$.
\end{Rem}

\section{The exterior corkscrew condition is sufficient for game-regularity} \label{corkpH} 

\begin{Def}\label{cork_def}
We say that a given boundary point $x_0\in \partial\mathcal{D}$
satisfies the {\em exterior corkscrew condition} provided
that there exists $\mu\in (0,1)$ such that for all sufficiently small
$r>0$ there exists a ball $B(x, {\mu r})$ such that:
$$B(x,{\mu r})\subset B(x_0,r)\setminus \bar{\mathcal{D}}.$$
\end{Def}

\smallskip

The main result of this section is:
\begin{Teo}\label{corkthengamereg}
If $x_0\in\partial\mathcal{D}$ satisfies the exterior corkscrew
condition, then $x_0$ is game-regular.
\end{Teo}

Towards the proof, we first recall a useful result on concatenating strategies, which
proposes a condition equivalent to the game-regularity criterion in
Definition \ref{def_gamereg} (a).  This result has been proved with little
detail in \cite{PS}, we thus reprove it for the convenience of the
reader.  Let $\mathcal{D}\subset\heis$ be an open, bounded,
connected domain.

\begin{Teo}\label{th_concat}
For a given $x_0\in\partial\mathcal{D}$, assume that there
exists $\theta_0\in (0,1)$ such that for every $\delta>0$ there exists
$\hat\delta\in (0, \delta)$ and  $\hat\epsilon\in (0,1)$ with the following property. Fix
$\epsilon\in (0, \hat\epsilon)$ and choose an initial position
$x_0\in B(x_0,{\hat\delta})$; there exists a strategy $\sigma_{0,II}$ such that for every $\sigma_I$ we have:
\begin{equation}\label{name33}
\mathbb{P}\big(\exists n< \tau \quad X_n\not\in B(x_0,\delta)\big)\leq\theta_0.
\end{equation}
Then $x_0$ is game-regular.
\end{Teo}
\begin{proof}
{\bf 1.} Under condition (\ref{name33}), construction of an optimal strategy realising the
(arbitrarily small) threshold $\eta$ in (\ref{gar}) is carried out by
concatenating the $m$ optimal strategies corresponding to the
achievable threshold $\eta_0$,
on $m$ concentric balls, where $(1-\eta_0)^m= 1-\theta_0^m\geq 1-\eta$.

Fix $\eta,\delta>0$. We want to find $\hat\epsilon$ and
$\hat\delta$ such that (\ref{gar}) holds.
Observe first that for $\theta_0\leq \eta$ the claim
follows directly from (\ref{name33}). In the general case, let $m\in \{2,3,\ldots\}$ be such that:
\begin{equation}\label{power}
\theta_0^m \leq \eta.
\end{equation}
Below we inductively define the radii $\{\delta_k\}_{k=1}^m$, together with the quantities
$\{\hat\delta(\delta_k)\}_{k=1}^m$, $\{\hat\epsilon(\delta_k)\}_{k=1}^m$ 
from the assumed condition (\ref{name33}). Namely,
for every initial position in $B(x_0, \hat\delta(\delta_k))$
in the Tug of War game with step less than
$\hat\epsilon(\delta_k)$, there exists a strategy $\sigma_{0, II, k}$
guaranteeing exiting $B(x_0,{\delta_k})$ (before the process
is stopped) with probability at most $\theta_0$. 
We set $\delta_m=\delta$ and find $\hat\delta(\delta_m)\in (0,\delta)$ and
$\hat\epsilon(\delta_m)\in (0,1)$, with the
indicated choice of the strategy $\sigma_{0,II,m}$. Decreasing the
value of $\hat\epsilon(\delta_m)$ if necessary, we then set:
$$ \delta_{m-1}=\hat\delta(\delta_m) - (1+\gamma_\p)\hat\epsilon(\delta_m)>0.$$
Similarly, having constructed $\delta_{k}>0$, we find
$\hat\delta(\delta_{k})\in (0,\delta_k)$ and
$\hat\epsilon(\delta_{k})\in (0, \hat\epsilon(\delta_{k+1}))$ and define: 
$$\delta_{k-1}=\hat\delta(\delta_{k}) - (1+\gamma_\p) \hat\epsilon(\delta_{k})>0.$$ 
Eventually, we call:
$$\hat\delta = \hat\delta(\delta_1), \qquad \hat\epsilon = \hat\epsilon(\delta_1).$$

To show that the condition of game-regularity at $x_0$ is
satisfied, we will concatenate the strategies
$\{\sigma_{0,II,k}\}_{k=1}^m$  by switching to $\sigma_{0,II,k+1}$
immediately after the token exits $B(x_0,{\delta_k})\subset B(x_0,{\hat\delta(\delta_{k+1})})$.
This construction is carried out in the next step.

\medskip 

{\bf 2.} Fix $y_0\in B(x_0, {\hat\delta})$ and let $\epsilon\in (0,
\hat\epsilon)$. Define the strategy $\sigma_{0,II}$:
$$\sigma_{0,II}^n = \sigma_{0,II}^n\big(x_0, (x_1, w_1, s_1,
t_1),\ldots , (x_n, w_n, s_n, t_n)\big) \quad \mbox{ for all } \; n\geq 0,$$
separately in the following two cases.

\smallskip

\underline{Case 1.} If $x_k\in B(x_0, {\delta_1})$ for all $k\leq n$, then we set: 
\begin{equation*}
\sigma_{0,II}^n = \sigma_{0,II,1}^n\big(x_0, (x_1, w_1, s_1, t_1),\ldots , (x_n, w_n, s_n, t_n)\big).
\end{equation*}

\smallskip

\underline{Case 2.} Otherwise, define:
\begin{equation*}
\begin{split} 
k &  \doteq k(x_0, x_1,\ldots, x_n) = \max\Big\{ 1\leq k\leq m-1;~ \exists
\;0\leq i\leq n~~ q_i\not\in B_{\delta_k}(q_0)\Big\}\\
i &  \doteq \min\Big\{ 0\leq i\leq n; ~ q_i\not\in B(x_0, {\delta_k})\Big\}.
\end{split}
\end{equation*}
and set:
\begin{equation*}
\sigma_{0,II}^n = \sigma_{0,II,k+1}^{n-i}\big(x_i, (x_{i+1}, w_{i+1}, s_{i+1},
t_{i+1}),\ldots, (x_n, w_n, s_n, t_n)\big).
\end{equation*}
It is not hard to check that each $\sigma_{0,II}^n:H_n\to
B(0,{1})\subset\R^N$ is Borel measurable, as required. 
Let $\sigma_I$ be now any opposing strategy. In the auxiliary Lemma
\ref{Fubini} below we will show that:
\begin{equation}\label{indurings}
\mathbb{P}\big(\exists n< \tau \quad  X_n\not\in B(q_0, {\delta_k})\big)
\leq\theta_0 \mathbb{P}\big(\exists n< \tau \quad X_n\not\in B(q_0,{\delta_{k-1}})\big)
\quad \mbox{for all } ~~k=2\ldots m,
\end{equation}
Consequently:
\begin{equation*}
\begin{split}
\mathbb{P}\big(\exists n< \tau \quad X_n\not\in B(q_0, {\delta})\big)
\leq\theta_0^{m-1}  \mathbb{P}\big(\exists n\leq \tau \quad x_n\not\in
B(q_0,{\delta_{1}})\big)\leq \theta_0^m,
\end{split}
\end{equation*}
which yields the result by (\ref{power}) and completes the proof.
\end{proof}

\bigskip

The inductive bound (\ref{indurings}) is quite
straightforward; we produce a precise argument for the sake of the
reader less familiar with probabilistic arguments:

\begin{Lemma}\label{Fubini}
In the context of the proof of Theorem \ref{th_concat}, we have (\ref{indurings}).
\end{Lemma}
\begin{proof}
{\bf 1.} Denote:
$$\tilde\Omega = \big\{\exists n\leq\tau ~~~ X_n\not\in
B(y_0, {\delta_{k-1}})\big\}\subset\Omega.$$
Since: $\mathbb{P}\big(\exists n\leq \tau ~~~ X_n\not\in B(x_0,{\delta_k})\big) 
\leq \mathbb{P}\big(\exists n\leq \tau ~~~ X_n\not\in B(x_0,
{\delta_{k-1}})\big)$, it follows that if
$\mathbb{P}(\tilde\Omega) = 0$ then (\ref{indurings}) holds trivially. 
For $\mathbb{P}(\tilde\Omega) > 0$, we define the probability space
$(\tilde\Omega,\tilde{\mathcal{F}}, \tilde{\mathbb{P}})$ by: 
$$\tilde{\mathcal{F}} = \big\{A\cap\tilde\Omega; ~~ A\in\mathcal{F}\big\} \quad \mbox{
  and } \quad \tilde{\mathbb{P}}(A) = \frac{\mathbb{P}(A)}{\mathbb{P}(\tilde\Omega)} \;\;\; \mbox{ for
  all } \; A\in\tilde{\mathcal{F}}.$$ 
Define also the measurable space $(\Omega_{fin}, \mathcal{F}_{fin})$, by setting
$\Omega_{fin} = \bigcup_{n=1}^\infty\Omega_n$ and by taking $\mathcal{F}_{fin}$ to
be the smallest $\sigma$-algebra containing $\bigcup_{n=1}^\infty\mathcal{F}_n$.
Finally, consider the random variables: 
\begin{equation*}
\begin{split}
& Y_1:\tilde\Omega\to\Omega_{fin}  \qquad \quad Y_1\big(\{(w_n, s_n,
t_n)\}_{n=1}^\infty\big)  = \{(w_n, s_n, t_n)\}_{n=1}^{\tau_k}\\
& Y_2:\tilde\Omega\to \Omega \qquad \qquad Y_2\big(\{(w_n, s_n,
t_n)\}_{n=1}^\infty\big)  = \{(w_n, s_n, t_n)\}_{n= \tau_k+1}^\infty, 
\end{split}
\end{equation*}
where $\tau_k$ is the following stopping time on $\tilde\Omega$:
$$\tau_k = \min\big\{n\geq 1; \quad X_n\not\in B(x_0, {\delta_{k-1}})\big\}.$$

We claim that $Y_1$ and $Y_2$ are independent. Indeed, given
$n,m\in\mathbb{N}$ and $A_1\in {\mathcal{F}}_{n}, A_2\in{\mathcal{F}}_m$:
\begin{equation*}
\begin{split}
& {\mathbb{P}}\big(Y_1\in A_1\big) = {\mathbb{P}}_n\Big( A_1\cap
\{\tau_k =n\}\cap\bigcap_{i<n} \big\{t_i\leq d_\epsilon(x_{i-1})\big\}\Big),\\
&{\mathbb{P}}\big(Y_2\in A_2\big) = {\mathbb{P}}(\tilde\Omega)\cdot {\mathbb{P}}_m(A_2)\\
&{\mathbb{P}}\big(\{Y_1\in A_1\}\cap \{Y_2\in A_2\}\big) = {\mathbb{P}}_n\Big( A_1\cap
\{\tau_k =n\}\cap\bigcap_{i<n} \big\{t_i\leq d_\epsilon(x_{i-1})\big\}\Big)\cdot {\mathbb{P}}_m(A_2).
\end{split}
\end{equation*}
This implies the following property equivalent to the claimed independence:
\begin{equation*}
 {\mathbb{P}}(\tilde\Omega)\cdot  {\mathbb{P}}\big(\{Y_1\in A_1\}\cap \{Y_2\in A_2\}\big)
  = {\mathbb{P}}\big(Y_1\in A_1\big) \cdot {\mathbb{P}}\big(Y_2\in A_2\big).
\end{equation*}
Consequently, Fubini's theorem yields for every random variable
$Z:\Omega_{fin}\times\Omega\to\bar{\mathbb{R}}_+$, that is measurable
with respect to the product $\sigma$-algebra of $\mathcal{F}_{fin}$ and $\mathcal{F}$:
\begin{equation}\label{fubi}
\int_{\tilde \Omega} Z\big(Y_1(\omega), Y_2(\omega)\big)\dd
\tilde{\mathbb{P}}(\omega) = \int_{\tilde \Omega} \int_{\tilde \Omega}
Z\big(Y_1(\omega_1), Y_2(\omega_2)\big)\dd
\tilde{\mathbb{P}}(\omega_2) \dd \tilde{\mathbb{P}}(\omega_1).
\end{equation}

\medskip

{\bf 2.} We now apply (\ref{fubi}) to the indicator random variable:
\begin{equation*}\label{Ffun}
Z\big(\{(w_i, s_i, t_i)\}_{i=1}^n, \{(w_i, s_i,
t_i)\}_{i=n+1}^\infty\big) = \mathbbm{1}_{\big\{ \mbox{\small ${\exists n\leq \tau \quad X_n(\{(w_i,
  s_i, t_i)\}_{i=1}^\infty)\not\in B(x_0,{\delta_k})}$}\big\}}, 
\end{equation*} 
to the effect that:
\begin{equation}\label{power2}
{\mathbb{P}}\big(\exists n\leq\tau \quad X_n\not\in B(x_0,{\delta_k})\big) =
\int_{\tilde\Omega} f(\omega_1)~\mbox{d}\tilde{\mathbb{P}}(\omega_1),
\end{equation}
where for a given $\omega_1 = \{(w_n, s_n,
t_n)\}_{n=1}^\infty\in\tilde\Omega$, the integrand function $f$ returns: 
\begin{equation*}
\begin{split}
f(\omega_1) & = {\mathbb{P}}\Big(\{(\bar w_n, \bar s_n,
\bar t_n)\}_{n=1}^\infty\in\tilde\Omega; \quad \exists
n\leq\tau \quad X_n\big(\{(w_i,  s_i, t_i)\}_{i=1}^{\tau_k}, \{(\bar
w_i,  \bar s_i, \bar t_i)\}_{i=\tau_k+1}^\infty\big) 
\not\in B(x_0,{\delta_k})\Big) \\ & = {\mathbb{P}}\Big(\{(\bar w_n, \bar s_n,
\bar t_n)\}_{n=1}^\infty\in\tilde\Omega; \quad \exists
n\leq\tau \quad X_n^{x_{\tau_k}, \sigma_I, \sigma_{0, II, k}}\big(\{(\bar w_i, \bar s_i, \bar
t_i)\}_{i=\tau_k+1}^\infty\big)  \not\in B(x_0, {\delta_k})\Big).
\end{split}
\end{equation*}
Since $x_{\tau_k}\in B(x_0,{\hat\delta(\delta_k)})$, by (\ref{name33}) it follows that:
\begin{equation*}
f(\omega_1) = {\mathbb{P}}\Big(\exists n\leq\tau \quad X_n^{x_{\tau_k}, \sigma_I, \sigma_{0, II, k}} \not\in
B(x_0, {\delta_k})\Big)\cdot {\mathbb{P}}(\tilde\Omega)\leq \theta_0\cdot
{\mathbb{P}}(\tilde\Omega) 
\end{equation*}
for $\tilde{\mathbb{P}}$-a.e. $ \omega_1\in \tilde\Omega$.
In conclusion, (\ref{power2}) implies (\ref{indurings}) and completes the proof.
\end{proof}

\medskip

The proof of game-regularity in Theorem \ref{corkthengamereg} will be based on
the concatenating strategies technique in the proof of Theorem \ref{th_concat} and
the analysis of the annulus walk below. Namely, one needs to derive 
an estimate on the probability of exiting a given 
annular domain $\tilde{\mathcal{D}}$ through the external portion of its
boundary. It follows \cite{PS} that when the ratio of the annulus thickness and the distance of the
initial token position from the internal boundary is large
enough, then this probability may be bounded by a universal constant $\theta_0<1$.
When $\p\geq N$, then $\theta_0$ converges to $0$ as the indicated
ratio goes to $\infty$.

\begin{Teo}\label{annulus}
For given radii $0<R_1<R_2<R_3$, consider the annulus
$\tilde{\mathcal{D}} = B(0,{ R_3})\setminus \bar B(0,{R_1})\subset
\R^N$. For every $\xi>0$, there exists $\hat\epsilon\in (0,1)$ depending on $R_1, R_2,
R_3$ and $\xi, \p, N$, such that for every $x_0\in \tilde{\mathcal{D}} \cap B(0,{R_2})$ and every
$\epsilon\in (0,\hat\epsilon)$, there exists a strategy $\tilde \sigma_{0,II}$
with the property that for every strategy $\tilde\sigma_I$ there holds:
\begin{equation}\label{name}
\mathbb{P}\Big(\tilde X_{\tilde\tau-1}\not\in \bar B(0,{R_3 -
  \epsilon})\Big)\leq \frac{v(R_2) - v(R_1)}{v(R_3) - v(R_1)} + \xi.
\end{equation}
Here, $v:(0,\infty)\to\mathbb{R}$ is given by:
\begin{equation}\label{vdef}
v(t) = \left\{\begin{array}{ll} \displaystyle{\mathrm{sgn}(\p-N) \,t^{\frac{\p-N}{\p-1}}} & \mbox{ for }
    \p\neq N\\ \log t & \mbox{ for } \p=N,\end{array}
\right.
\end{equation}
and $\{\tilde X_n = \tilde X_n^{\epsilon, x_0, \tilde\sigma_I,
  \tilde\sigma_{0,II}}\}_{n=0}^\infty$ and $\tilde\tau =
\tilde\tau^{\epsilon, x_0, \tilde\sigma_I,
  \tilde\sigma_{0,II}}$ denote, as before, the random variables corresponding
to positions and stopping time in the random Tug of War game on $\tilde{\mathcal{D}}$.
\end{Teo}
\begin{proof}
Consider the radial function $u:\R^N\setminus\{0\}\to\mathbb{R}$ given by
$u(x) = v(|x|)$, where $v$ is as in (\ref{vdef}). Recall that:
\begin{equation}\label{vdef2}
\Delta_{\p} u = 0 \quad \mbox{ and } \quad \nabla u\neq
0 \qquad \mbox{ in } \,\,\R^N\setminus\{0\}. 
\end{equation}
Let $\tilde u_\epsilon$ be the family of solutions to (\ref{DPP2}) with the  data $F$ 
provided by a smooth and bounded modification of $u$ outside of the
annulus $B(,{2R_3})\setminus \bar{B}(0,\frac{R_1}{2})$.
By Theorem \ref{conv_nondegene}, there exists a constant $C>0$, depending
only on $\p, u$ and $\tilde{\mathcal{D}}$, such that:
$$\|\tilde u_\epsilon - u\|_{\mathcal{C}(\tilde{\mathcal{D}})}\leq C\epsilon
\qquad \mbox{as }\;\epsilon\to 0.$$

For a given $x_0\in \tilde{\mathcal{D}}\cap B_{R_2}(0)$, 
there exists thus a strategy $\tilde\sigma_{0,II}$ so that for every $\tilde\sigma_{I}$ we have: 
\begin{equation}\label{vdef3} 
\mathbb{E}\big[ u\circ (\tilde X^{\epsilon, x_0,
  \tilde\sigma_I, \tilde\sigma_{0, II}})_{\tilde \tau-1}\big] - u(x_0) \leq 2C\epsilon.
\end{equation}
We now estimate:
\begin{equation*}
\begin{split}
\mathbb{E} \big[u\circ  \tilde  X_{\tilde \tau-1}\big] - u(x_0) & = 
\int_{\{\tilde X_{\tilde\tau-1}\not\in \bar B(0, R_3 - \epsilon)\}} u(\tilde
X_{\tilde \tau-1})~\mbox{d}\mathbb{P} + 
\int_{\{\tilde X_{\tilde\tau-1}\in B(0, {R_1 + \epsilon})\}} u(\tilde
X_{\tilde \tau-1})~\mbox{d}\mathbb{P} - u(x_0) \\ & \geq \mathbb{P}\big(\tilde
X_{\tilde\tau-1}\not\in \bar B(0,{R_3-\epsilon})\big)
v\big(R_3- \epsilon\big) \\ & \qquad + \Big(1- \mathbb{P}\big(\tilde
X_{\tilde\tau-1}\not\in \bar B(0,{R_3 - \epsilon})\big)\Big)
v\big(R_1 - \gamma_\p\epsilon\big) - v(R_2),
\end{split}
\end{equation*}
where we used the fact that $v$ in (\ref{vdef}) is an increasing
function. Recalling (\ref{vdef3}), this implies:
\begin{equation}\label{vdef4}
\mathbb{P}\big(\tilde X_{\tilde\tau-1}\not\in \bar B(0,{R_3 - \epsilon})\big)
\leq \frac{ v(R_2) -
  v(R_1 - \gamma_\p\epsilon) + 2C\epsilon}{v(R_3 - \epsilon) -
  v(R_1 - \gamma_\p\epsilon)}.
\end{equation}
The proof of (\ref{name}) is now complete, by continuity of the right hand side
with respect to $\epsilon$.
\end{proof}

\medskip

By inspecting the quotient in the right hand side of (\ref{name}) we obtain:

\begin{Cor}\label{annulus3}
The function $v$ in (\ref{vdef}) satisfies, for any fixed $0<R_1<R_2$:
\begin{itemize}
\item[(a)] $~~~\displaystyle{\lim_{R_3\to \infty}\,\frac{v(R_2) - v(R_1)}{v(R_3) -
  v(R_1)} = \left\{\begin{array}{ll}
\displaystyle{1- \Big(\frac{R_2}{R_1}\Big)^{\frac{\p-N}{\p-1}}} & \mbox{for } 1<\p<N\vspace{1mm}\\
0 & \mbox{for } \p\geq N, \end{array}\right.\displaystyle}$
\item[(b)] $~~~\displaystyle{\lim_{M\to \infty}\,\frac{v(MR_1) - v(R_1)}{v(M^2R_1) -
  v(R_1)} = \left\{\begin{array}{ll} \displaystyle{\frac{1}{2}} & \mbox{for } \p=N\vspace{1mm}\\
 \,0 & \mbox{for } \p>N. \end{array}\right.\displaystyle}$
\end{itemize}
Consequently, the estimate (\ref{name}) can be replaced by:
\begin{equation}\label{name2}
\mathbb{P}\big(\tilde X_{\tilde\tau-1}\not\in \bar B(0,{R_3 -
  \epsilon})\big)\leq \theta_0
\end{equation}
valid for any $\theta_0> 1- \big(\frac{R_2}{R_1}\big)^{\frac{\p-N}{\p-1}}$
if $\p\in (1, N)$, and any $\theta_0>0$ if $\p\geq N$, upon choosing $R_3$
sufficiently large with respect to $R_1$ and $R_2$. Alternatively,
when $\p>N$, the same bound with arbitrarily small $\theta_0$ can be achieved by
setting $R_2=MR_1$, $R_3=M^2R_1$, with $M$ large enough.
\end{Cor}

\medskip

The results of Theorem \ref{annulus} and Corollary \ref{annulus3} are
invariant under scaling, i.e.:

\begin{Rem}\label{annulus2}
The bounds (\ref{name}) and (\ref{name2}) remain true if we replace
$R_1, R_2$, $R_3$ by $r R_1, r R_2$, $r R_3$, the domain
$\tilde{\mathcal{D}}$ by $r\tilde{\mathcal{D}}$ and $\hat\epsilon$
by $r\hat\epsilon$, for any $r>0$.
\end{Rem}

\medskip

\begin{figure}[htbp]
\centering
\includegraphics[scale=0.4]{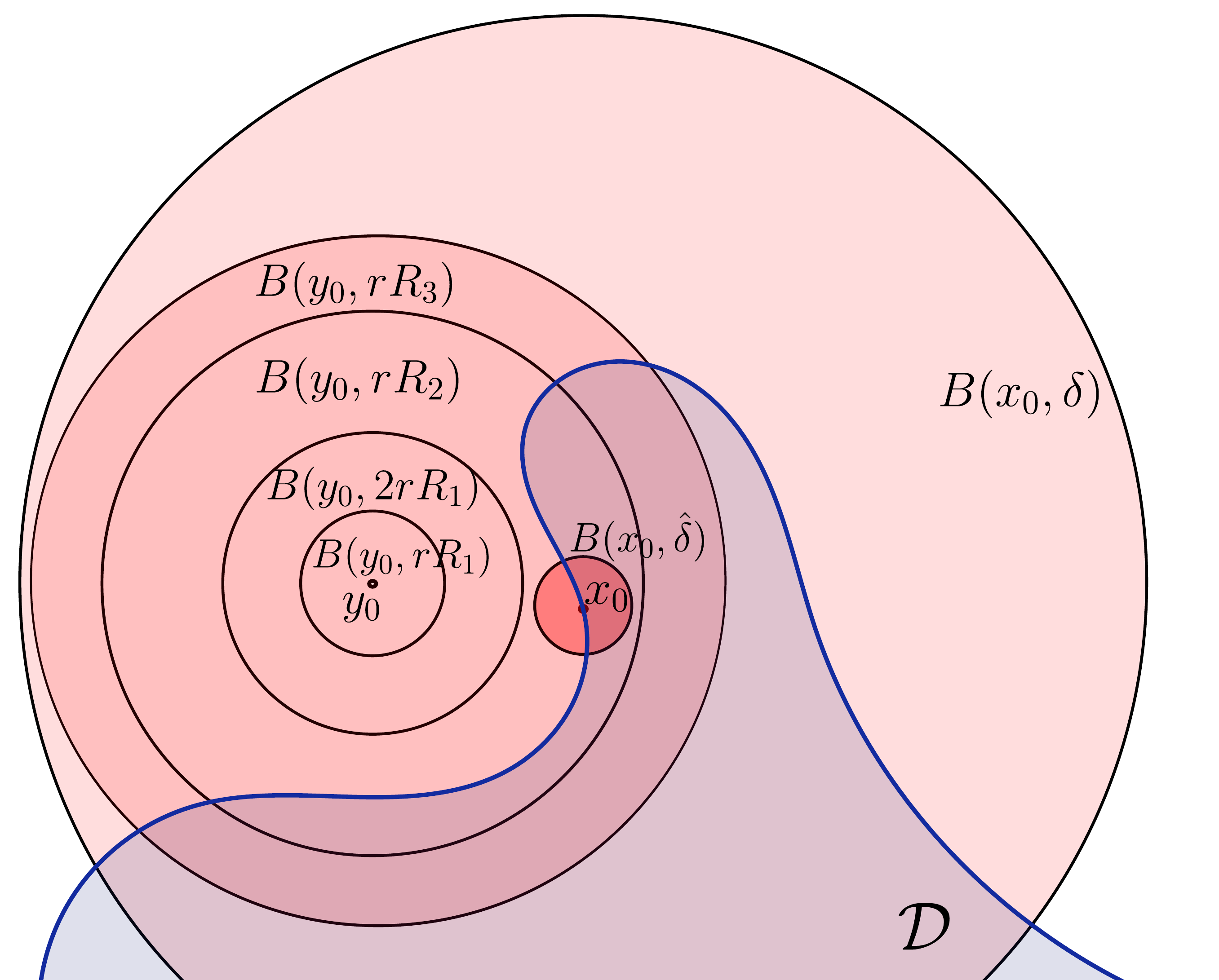}
    \caption{{Positions of the concentric balls $B(y_0, \cdot)$ and
        $B(x_0, \cdot)$ in the proof of Theorem \ref{corkthengamereg}.}}
\label{f:cork_proof}
\end{figure}

\bigskip

\noindent {\bf Proof of Theorem \ref{corkthengamereg}.}

\smallskip

\noindent With the help of Theorem \ref{annulus}, we will show that the
assumption of Theorem \ref{th_concat} is satisfied,
with proba\-bi\-li\-ty $\theta_0<1$ 
depending only on $\p, N$ and $\mu\in (0,1)$ in Definition \ref{cork_def}. Namely,
set $R_1=1$, $R_2= \frac{2}{\mu}$ and $R_3>R_2$ according to Corollary
\ref{annulus3} (a) in order to have $\theta_0= \theta_0(\p, N, R_1,
R_2)<1$. Further, set $r=\frac{\delta}{2R_3}$ so that
$rR_2=\frac{\delta}{\mu R_3}$. Using the corkscrew condition, we obtain:
$$B(y_0,{2rR_1}) \subset B(x_0, \frac{\delta}{\mu R_3})\setminus \bar{\mathcal{D}},$$
for some $y_0\in\R^N$.
In particular: $|x_0-y_0|<rR_2$, so $x_0\in B(y_0, {rR_2})\setminus
\bar B(y_0, {2rR_1})$. It now easily follows that there exists
$\hat\delta\in (0,\delta)$ with the property that:
$$B(x_0, {\hat\delta}) \subset B(y_0, {rR_2})\setminus \bar B(y_0, {2rR_1}).$$
Finally, we observe that $B(y_0, {rR_3})\subset B(x_0, \delta)$ because
$rR_3 + |x_0-y_0| < rR_3 + rR_2<2rR_3 = \delta$.

\medskip

Let $\hat\epsilon/r>0$ be as in Theorem \ref{annulus}, applied to the
annuli with radii $R_1, R_2, R_3$, in view of Remark \ref{annulus2}. For a given
$x\in B(x_0, {\hat\delta})$ and $\epsilon\in (0,\hat\epsilon)$, let
$\tilde\sigma_{0,II}$ be the strategy ensuring validity of the bound
(\ref{name2}) in the annulus walk on $y_0 + \tilde{\mathcal{D}}.$ 
For a given strategy $\sigma_I$ there holds:
\begin{equation*}
\begin{split}
\Big\{\omega\in\Omega; ~  \exists n<\tau^{\epsilon, x, \sigma_I,
  \sigma_{0,II}}&(\omega) \qquad X_n^{\epsilon,x, \sigma_I,
  \sigma_{0,II}}(\omega)\not\in B(x_0, \delta)\Big\} \\ & \subset 
\Big\{\omega\in\Omega; ~ \tilde X^{\epsilon, x, \tilde \sigma_I,
  \tilde \sigma_{0,II}}_{\tilde\tau-1}(\omega)\not\in B(y_0, {r R_3-\epsilon})\Big\}.
\end{split}
\end{equation*}
The final claim follows by (\ref{name2}) and by applying  Theorem \ref{th_concat}.
\endproof

\begin{Rem}
Using Corollary \ref{annulus3} (b) one can show that every open,
bounded domain $\mathcal{D}\subset\R^N$ is game-regular for $\p>N$. 
The proof mimics the argument of \cite{PS} for the process based on
the mean value expansion (\ref{PS_dpp}), so we omit it.
\end{Rem}

\section{Uniqueness and identification of the limit in Theorem
  \ref{thm_gamereg_conv}} 

Let $F\in\mathcal{C}(\R^N)$ be a bounded data function and let
$\mathcal{D}$ be open, bounded and game-regular. In virtue of Theorem
\ref{thm_gamereg_conv} and the Ascoli-Arzela theorem, every sequence in
the family $\{u_\epsilon\}_{\epsilon\to 0}$ of solutions to
(\ref{DPP2}) has a further subsequence converging uniformly to some
$u\in\mathcal{C}(\R^N)$ and satisfying $u=F$ on $\R^N\setminus \mathcal{D}$.
We will show that such limit $u$ is in fact unique.

\medskip

Recall first the definition of the $\p$-harmonic viscosity solution:

\begin{Def}\label{visco_def_p}
We say that $u\in\mathcal{C}(\bar{\mathcal{D}})$ is a {\em viscosity solution} to the problem:
\begin{equation}\label{problem} 
\Delta_{\p}u =0 \quad \mbox{in }\mathcal{D}, \qquad u=F \quad \mbox{on }\partial\mathcal{D},
\end{equation}
if the latter boundary condition holds and if:
\begin{itemize}
\item[(i)] for every $x_0\in\mathcal{D}$ and every $\phi\in\mathcal{C}^2(\bar{\mathcal{D}})$ such that:
\begin{equation}\label{assu_super}
\phi(x_0) = u(x_0), \quad \phi<u ~~ \mbox{in }
\bar{\mathcal{D}}\setminus \{x_0\}\quad \mbox{ and } \quad  \nabla \phi(x_0)\neq 0,
\end{equation}
there holds: $\Delta_{\p} \phi(x_0)\leq 0$,
\item[(ii)] for every $x_0\in\mathcal{D}$ and every $\phi\in\mathcal{C}^2(\bar{\mathcal{D}})$ such that:
\begin{equation*}\label{assu_sub}
\phi(x_0) = u(x_0), \quad \phi>u ~~ \mbox{in }
\bar{\mathcal{D}}\setminus \{x_0\}\quad \mbox{ and } \quad  \nabla \phi(x_0)\neq 0,
\end{equation*}
there holds: $\Delta_{\p} \phi(x_0)\geq 0$.
\end{itemize}
\end{Def}

\begin{Teo}\label{unif-topharm}
Assume that the sequence $\{u_\epsilon\}_{\epsilon\in J, \epsilon\to 0}$ of solutions to
(\ref{DPP2}) with a bounded data function $F\in\mathcal{C}(\R^N)$,
converges uniformly as $\epsilon\to 0$ to some limit  
$u\in\mathcal{C}(\R^N)$. Then $u$ must be the viscosity solution to (\ref{problem}).
\end{Teo}
\begin{proof}
{\bf 1.} Fix $x_0\in\mathcal{D}$ and let $\phi$ be a test function as in (\ref{assu_super}).
We first claim that there exists a sequence
$\{x_\epsilon\}_{\epsilon\in J}\in\mathcal{D}$, such that:
\begin{equation}\label{4.7}
\lim_{\epsilon\to 0, \epsilon\in J}x_\epsilon = x_0 \quad \mbox{ and } \quad
u_\epsilon(x_\epsilon) - \phi(x_\epsilon) = \min_{\bar{\mathcal{D}}} \,(u_\epsilon - \phi).
\end{equation}
To prove the above, for every $j\in\mathbb{N}$ define $\eta_j>0$ and $\epsilon_j>0$ such that:
$$\eta_j = \min_{\bar{\mathcal{D}}\setminus B(x_0,{\frac{1}{j}})} (u-\phi) \qquad\mbox{and}
\qquad \|u_\epsilon - u\|_{\mathcal{C}(\bar{\mathcal{D}})} \leq \frac{1}{2}\eta_j \quad
\mbox{ for all } \epsilon\leq \epsilon_j.$$
Without loss of generality, the sequence $\{\epsilon_j\}_{j=1}^\infty$ is decreasing to
$0$ as $j\to\infty$. Now, for $\epsilon\in (\epsilon_{j+1},
\epsilon_j]\cap J$, let $x_\epsilon\in \bar B(x_0, {\frac{1}{j}})$ satisfy: 
$$u_\epsilon(q_\epsilon) - \phi(q_\epsilon) = \min_{\bar{B}(x_0,{\frac{1}{j}})}(u_\epsilon - \phi).$$
Observing that the following bound is valid for every $q\in
{\bar{\mathcal{D}}\setminus B(x_0,{\frac{1}{j}})}$, proves (\ref{4.7}):
\begin{equation*}
\begin{split}
u_\epsilon(q) - \phi(q) & \geq u(q) - \phi(q) - \|u_\epsilon -
u\|_{\mathcal{C}(\bar{\mathcal{D}})}\geq \eta_j - \frac{1}{2}\eta_j \geq \|u_\epsilon
- u\|_{\mathcal{C}(\bar{\mathcal{D}})} \\ & \geq u_\epsilon(q_0)- \phi(q_0)\geq
\min_{\bar{B}(x_0,{\frac{1}{j}})}(u_\epsilon - \phi).
\end{split}
\end{equation*}

\smallskip

{\bf 2.} Since by (\ref{4.7}) we have: $\phi(x) \leq u_\epsilon(x) +
\big(\phi(x_\epsilon) - u_\epsilon(x_\epsilon)\big)$ for all $x\in\bar{\mathcal{D}}$, it follows that:
\begin{equation}\label{4.8}
S_\epsilon \phi(x_\epsilon) - \phi(x_\epsilon) \leq  S_\epsilon u(x_\epsilon)
+ \big(\phi(x_\epsilon) - u_\epsilon(x_\epsilon)\big) - \phi(x_\epsilon)  = 0,
\end{equation}
for all $\epsilon$ small enough to guarantee that $d_\epsilon(x_\epsilon) = 1$.
On the other hand, (\ref{expansion}) yields:
\begin{equation*}
S_\epsilon \phi(x_\epsilon)  - \phi(x_\epsilon) = \frac{\epsilon^2}{\p-1}|\nabla\phi(x_\epsilon)|^{2-\p}
\Delta_{\p}\phi(x_\epsilon) + o(\epsilon^2),
\end{equation*}
for $\epsilon$ small enough to get $\nabla\phi(x_\epsilon)\neq 0$.
Combining the above with (\ref{4.8}) gives:
$$\Delta_{\p}\phi(x_\epsilon) \leq o(1).$$
Passing to the limit with $\epsilon\to 0, \epsilon\in J$ establishes the desired
inequality $\Delta_{\p}\phi(x_0)\leq 0$ and proves part (i) of Definition  \ref{visco_def_p}.
The verification of part (ii) is done along the same lines.
\end{proof}

Since the viscosity solutions $u\in\mathcal{C}(\bar{\mathcal{D}})$ of (\ref{problem}) are
unique (see a direct proof in \cite[Appendix, Lemma
4.2]{LM}), in view of Theorem \ref{unif-topharm} and Theorem 
\ref{thm_gamereg_conv} we obtain:

\begin{Cor}
Let $F\in\mathcal{C}(\R^N)$ be a bounded data function and let
$\mathcal{D}$ be open, bounded and game-regular. The family
$\{u_\epsilon\}_{\epsilon\to 0}$ of solutions to 
(\ref{DPP2}) converges uniformly in $\bar{\mathcal{D}}$ to the unique
viscosity solution of (\ref{problem}).
\end{Cor}

\end{document}